\documentclass{amsart}

\usepackage{geometry}
\usepackage{cite}
\usepackage{pdfpages}
\usepackage{xcolor}
\usepackage[colorlinks=true,linkcolor=blue]{hyperref}

\usepackage{soul}

\newtheorem{thm}{Theorem}[section]

\newtheorem{prop}{Proposition}[section]
\newtheorem{lem}[prop]{Lemma}

\theoremstyle{definition}

\newtheorem{remark}[prop]{Remark}
\newtheorem{definition}[prop]{Definition}

\DeclareMathOperator{\Wr}{\operatorname{Wr}}

\newcommand{\lp}{\left(}
\newcommand{\rp}{\right)}
\newcommand{\lb}{\left[}
\newcommand{\rb}{\right]}

\newcommand{\al}{\alpha}

\newcommand{\hT}{\hat{T}}
\newcommand{\hP}{\hat{P}}
\newcommand{\hA}{B}
\newcommand{\hb}{\hat{b}}

\newcommand{\hq}{\hat{q}}
\newcommand{\hr}{\hat{r}}
\newcommand{\hw}{\hat{w}}
\newcommand{\htau}{\hat{\tau}}

\newcommand{\hrho}{\hat{\rho}}

\newcommand{\supth}{^{\text{th}}}

\newcommand{\hlambda}{\hat{\lambda}}
\newcommand{\hpi}{\hat{\pi}}
\newcommand{\tpi}{\tilde{\pi}}

\newcommand{\tW}{\tilde{W}}
\newcommand{\tnu}{\tilde{\nu}}

\newcommand{\psiperp}{\psi^{\perp}}
\newcommand{\phiperp}{\phi^{\perp}}

\newcommand{\ttau}{\tilde{\tau}}
\newcommand{\trho}{\tilde{\rho}}

\newcommand{\tA}{\tilde{A}}

\newcommand{\tT}{\tilde{T}}

\newcommand{\tq}{\tilde{q}}
\newcommand{\tr}{\tilde{r}}
\newcommand{\tw}{\tilde{w}}
\newcommand{\tpsi}{\tilde{\psi}}

\newcommand{\N}{\mathbb{N}}
\newcommand{\Nz}{\N_0}

\newcommand{\R}{\mathbb{R}}

\newcommand{\bt}{\boldsymbol{t}}
\newcommand{\bm}{{\boldsymbol{m}}}

\newcommand{\bQ}{{\boldsymbol{Q}}}

\newcommand{\cR}{{\mathcal{R}}}

\newcommand{\cT}{\mathcal{T}}

\newcommand{\Cal}{C^{(\alpha)}}

\newcommand{\taual}{\tau^{(\alpha)}}
\newcommand{\cRal}{\cR^{(\alpha)}}
\newcommand{\Pal}{P^{(\al)}}
\newcommand{\Wal}{W^{(\al)}}
\newcommand{\nual}{\nu^{(\al)}}
\newcommand{\rhoal}{\rho^{(\al)}}
\newcommand{\Ral}{R^{(\alpha)}}
\newcommand{\Tal}{T^{(\alpha)}}

\newcommand{\bQal}[1]{\bQ^{(\alpha)}_{#1}}
\newcommand{\Bal}{B^{(\alpha)}}

\newcommand{\rL}{\mathrm{L}}

\newcommand{\tor}[1]{{\color{blue} #1}}

\title{Exceptional Gegenbauer polynomials via isospectral deformation}

\author{Mar\'ia~\'Angeles Garc\'ia-Ferrero}
\address{BCAM - Basque Center for Applied Mathematics, Alameda de Mazarredo 14, 48009 Bilbao, Spain}
\email{mgarcia@bcamath.org}
  
\author{David G\'omez-Ullate}
\address{Departamento de Ingenier\'ia Inform\'atica, Escuela Superior de Ingenier\'ia, Universidad de C\'adiz, 11519 Puerto Real, Spain.}
\address{Departamento de F\'isica Te\'orica, Universidad Complutense de Madrid, 28040 Madrid, Spain.}
\email{david.gomezullate@uca.es}

\author{Robert Milson}
\address{Department of Mathematics and Statistics, Dalhousie University, Halifax, NS, B3H 3J5, Canada.}
\email{rmilson@dal.ca}

\author{James Munday}
\address{Department of Mathematics and Statistics, Dalhousie University, Halifax, NS, B3H 3J5, Canada.}
\email{james.munday@dal.ca}

\date{Updated September 29, 2021}

\begin{document}

\begin{abstract}
We show a method to construct isospectral deformations of classical orthogonal polynomials.
The construction is based on confluent Darboux transformations, and it allows to construct Sturm-Liouville problems with polynomial eigenfunctions that have an arbitrary number of continuous parameters. We propose to call these new orthogonal polynomial systems \emph{exceptional polynomials of the second kind}.
We illustrate this construction by describing the class of exceptional Gegenbauer polynomials of the second kind.
\end{abstract}

\maketitle
\setcounter{tocdepth}{1}
\tableofcontents

\section{Introduction}

Classical orthogonal polynomials (OPs) have been traditionally characterized as the only orthogonal polynomial bases of an $\rm L^2$ space that are also eigenfunctions of a Sturm-Liouville problem. Since the mid-nineteenth century until today, classical  OPs appear in ubiquitous applications in mathematical  physics,  numerical analysis, approximation theory or statistics, among other fields.

As it is well known, classical OPs can be classified into three main families, depending on whether they are defined on the whole real line (Hermite), the half-line (Laguerre) or a compact interval (Jacobi). These polynomial families are characterized by two free real parameters in the case of Jacobi, one for Laguerre and none for Hermite. 

This article tackles the following question: \emph{is it possible to construct an orthogonal basis of an $\rm L^2$ space, which is also formed by polynomial eigenfunctions of a Sturm-Liouville problem, but contains a higher (possibly arbitrary) number of free real parameters?}

The flexibility of deforming classical OPs to contain so many free real parameters and yet mantain many of their defining properties would open the way to many potential applications in all of the fields where classical OPs appear naturally.

In this paper we show how to construct Sturm-Liouville problems with polynomial eigenfunctions that contain an arbitrary number of real parameters, thus providing a positive answer to the previous question.
More specifically, let $T$ be a \emph{classical differential operator}, i.e. a second-order differential operator whose eigenfunctions are classical OPs, 
\[ T(z,D_z):=p(z)D_z^2 + q(z)D_z + r(z),\]
and let $\{P_i\}_{i=0}^\infty$ be a set of polynomial eigenfunctions of $T$ 
\[ T P_i=\lambda_i P_i,\quad i=0,1,2,\dots\] We will show how to
construct a new operator $\hT$ and polynomials $\hP_i$ such that
\[ \hT(z,D_z;t_{m_1},\dots,t_{m_n}):=p(z)D_z^2 + \hq(z;
  t_{m_1},\dots,t_{m_n} )D_z + \hr(z;t_{m_1},\dots,t_{m_n})\]
where $t_{m_1},\dots,t_{m_n}$ are $n$ real parameters, $\hq$ and $\hr$
are rational functions of $z$, and
\[ \hT \hP_i=\lambda_i \hP_i,\quad i\in\{0,1,2,\dots\}.\] 

The transformed eigenvalue problem has the same spectrum
$\{ \lambda_i\}_{i=0}^\infty$ as the original one, and the leading
coefficient $p(z)$ does not change in the transformation. We speak
thus of a \emph{isospectral deformation} because we will also see that
\[ \hT(z,D_z;0,\dots,0)=T(z,D_z),\qquad
  \hP_i(z;0,\dots,0)=P_i(z).\] So much extra freedom and
flexibility comes at a cost. Although the set of new polynomials is
still $\rL^2$-complete, there is a finite number of exceptional
degrees for which no polynomial of that degree exists in the basis. In
general we will have that
\[i=\deg P_i \neq \deg \hP_i,\quad i\in\{0,1,2,\dots\}\]

The new polynomial families introduced in this paper fit thus into the definition of \emph{exceptional orthogonal polynomials} (XOPs) originally introduced in \cite{GKM09}. In the past ten years,  the mathematical study of XOPs has been
concerned with their classification \cite{GFGM19, GKM12}, the
properties of their zeroes \cite{GMM13, Ho15, KM15}, and their
recurrence relations \cite{Du15, GKKM16, MT15, Od16}. The spectral-theoretic aspects of exceptional operators (the
second-order differential operators whose eigenfunctions are the
XOPs) have been studied in \cite{LL15,LL16,GGM21}.
It is known that every exceptional operator can be related to a
classical Bochner operator by a finite number of state-adding,
state-deleting, or isospectral Darboux transformations \cite[Theorem
1.2]{GFGM19}. 

In the context of mathematical physics, XOPs appear as exact
solutions to Dirac's equation \cite{SHR14} and as bound states of
exactly-solvable rational extensions \cite{GGM14, OS09,
  PTV12}. Additionally, they are connected to
finite-gap potentials \cite{HV10} and super-integrable systems
\cite{MQ13, PTV12}.

The contribution of this paper is to present a new, different construction of exceptional orthogonal polynomials, that we propose to call (rather unimaginatively) \textit{exceptional orthogonal polynomials of the second kind}. There are a number of important  differences between the class of known XOPs (of the first kind)  and the new XOPs of the second kind, including:
\begin{enumerate}
\item XOPs of the first kind  have the same number  of real parameters as their classical counterparts, while XOPs of the second kind have as many real parameters as  desired.
\item XOPs of the first kind are not a continuous deformation of a classical polynomial family, while XOPs of the second kind become classical when we set all deformation parameters to zero.
\item XOPs of the first kind have a different spectrum with respect to their classical counterparts while XOPs of the second kind have the same spectrum.
\item The degree sequence for XOPs of the first kind is an increasing sequence, which is no longer true for XOPs of the second kind.
\end{enumerate}


As an illustration of this new construction, we show how to define an isospectral
deformation of the classical Gegenbauer operator 
\begin{equation}
    \label{eq:T0def}
    \Tal := (1-z^2) D_z^2 - (2\al+1) z D_z 
\end{equation}
through the application of a finite number of confluent Darboux
transformations (CDTs), also referred to as the ``double
commutator'' method \cite{GT96}.

Exceptional Jacobi polynomials of the first kind are indexed by discrete parameters \cite{Bo20,Du17} and, as a result, cannot be continuously deformed into their classical
counterparts. By contrast, every CDT introduces a new
deformation parameter. Therefore, by performing a chain of $n$ CDTs on
the classical Gegenbauer operator \eqref{eq:T0def} at distinct
eigenvalues, we will arrive at an exceptional operator that depends on $n$ discrete parameters and $n$ real parameters. The eigenpolynomials of the resulting exceptional
operator are exceptional Gegenbauer polynomials of the second kind, which also depend on $n$ real parameters, and can be continuously deformed to the
classical Gegenbauer polynomials by letting the parameters tend to zero.

The new construction of exceptional polynomials and weights described in this paper can also be understood from the point of view of the theory of inverse scattering
\cite{AM80, CASH17, Su85}, and it is conceptually related to the
construction of KdV multi-solitons. While KdV solitons are obtained by applying a
state-adding deformation on the zero potential, the exceptional
operators in this paper are related to isospectral deformations of
particular instances of the Darboux--P\"oschl--Teller (DBT) potential
\cite{GB13}.  In spectral-theoretic terms, the consequence of our
construction is the continuous modification of the norming constants
of a finite number of the bound states.  However, since our focus
is on orthogonal polynomials, rather than quantum mechanics or
evolution equations, our approach is formulated in the gauge and coordinates of
the classical differential operator rather than working with Schr\"odinger
operators.  The resulting procedure can be easily implemented using a
computer algebra system.

For the remainder of the paper, we develop the theory of confluent Darboux transformations in the algebraic gauge, and
apply it in a construction of new exceptional Gegenbauer polynomials of the second kind. The paper is organized as follows:
in Section ~\ref{sec:algDar} we describe the formal theory of rational multi-step Darboux-Crum transformations in the algebraic gauge. In Section~\ref{sec:CDT} we describe confluent Darboux transformations as a 2-step Darboux transformation whose seed functions are an eigenfunction and a generalized eigenfunction at the same eigenvalue.
 In Section ~\ref{sec:Gegen}, we describe the class of exceptional Gegenbauer operators and their factorizations and we provide a recursive
construction for the operators and eigenfunctions connected by CDTs. In Section~\ref{sec:xgegen} we provide matrix formulas for exceptional Gegenbauer polynomials of second type, we prove the equivalence of
the matrix and recursive definitions and thereby establish the proofs
of our main results concerning the Sturm--Liouville properties and
$\rL^2$-completeness of the exceptional Gegenbauer
polynomials. Finally, in Section ~\ref{sec:examples}, we provide explicit
examples of families of exceptional Gegenbauer polynomials of the second kind, for both one and two deformation parameters.

Although this is the first paper to provide a detailed description of the construction of exceptional polynomials of the second kind via confluent Darboux transformations, we would like to mention that exceptional polynomials obtained by confluent Darboux transformations were investigated earlier in the 1-step case (with one deformation parameter) by Grandati and Quesne \cite{GQ15}.  The family of exceptional Legendre polynomials of the second kind  was recently introduced by some of the present authors \cite{GFGM21}, showing for the first time that deformations of classical polynomials with an arbitrary number of real parameters exist. Shortly after that, Dur\'an has shown that there is an alternative way to construct exceptional polynomials of the second kind, by first perturbing the  measure of the discrete Hahn polynomials, dualizing and taking a suitable limit\cite{Du21}. The extra continuous parameters enter in this construction in the polynomial perturbation of the classical discrete measure. Summarizing, there are three rather different constructions leading to the same mathematical objects:
\begin{enumerate}
\item[(i)] one based in iterating the action of differential operators, leading to Wronskian determinants whose seed functions  are generalized eigenfunctions, 
\item[(ii)] one based in matrix  and integral formulas, coming from the inverse scattering method,
\item[(iii)] one based in Dur\'an's bispectral construction, coming from a perturbation of a classical discrete measure.
\end{enumerate}

In this paper we show the equivalence of the first two approaches. Although there is a  strong evidence that the last construction leads to the same objects, establishing the equivalence of all three methods remains an open question, which can possibly trigger  further research.

\subsection*{Notation}
\label{sec:notation}


Throughout the paper we use $\N = \{1,2,\ldots \}$ to denote the set of natural
numbers and $\Nz = \{0, 1,\dots\}$ to denote the set of non-negative
integers.  We use ``half-integer'' to refer to an odd integer divided
by $2$.  The set of positive half-integers will be denoted by
$\Nz + \frac12$.

We let $D_z$ denote the derivative with respect to $z$.  For the sake
of notational convenience, we will often drop the explicit dependence
on the indeterminate $z$ and write
$\phi=\phi(z),\; \phi' =  D_z \phi =\phi_z$, and $D=D_z$.

We call a differential expression of the form
$\sum_{k=0}^n p_k(z) D_z^k,$ where
$p_0(z),\ldots, p_n(z)$ are rational functions and $p_n \ne 0$,
an \emph{$n\supth$-order rational
  operator}.  We will call a function $\phi(z)$ quasi-rational if its
log-derivative $w(z) = \phi'(z)/\phi(z)$ is rational.  
We denote by $\Wr[y_1, \dots, y_k]$ the Wronskian determinant of the functions $y_1, \dots, y_k$.

We will denote matrices by calligraphic symbols, such as $\cR$,
whereas 1-dimensional tuples will be given bold symbols such as $\bm$,
$\bt$, or $\bQ$. To access the components of a vector or tensor we
will employ square brackets, i.e. $[\cR]_{k\ell}$ denotes the
$(k,\ell)$ entry of $\cR$.

We will denote an $n$--tuple of integers by
$\bm=(m_1, \dots, m_n)\in\Nz^n$, and associate to it the $n$--tuple of
real parameters $\bt_{\bm}=(t_{m_1}, \dots, t_{m_n})\in \R^n$. We will
separate objects of different natures, such as real parameters and
tuples, by semicolons. The concatenation of tuples will be shown using
commas, e.g. if $i_1, \dots, i_k \in \Nz$, then
$(\bm, i_1, \dots, i_k)$ denotes the $(n+k)$--tuple
$(m_1, \dots, m_n, i_1, \dots, i_k)$. We will frequently omit
parentheses when denoting $1$-tuples, opting to write $m_1$ instead of
$(m_1)$. Occasionally, we will omit the dependence on the parameter
$\al$, so as not to conflict with other superscript notations.





\section{Darboux transformations and factorization chains}
\label{sec:algDar}

The formal theory of Darboux transformations for Sch\"odinger operators in mathematical physics (also known as supersymmetric quantum mechanics) has been developed in numerous works, mostly with the aim of generating new solvable problems from known ones, \cite{cooper,CASH17,GKM04}, but also in the construction of solutions to Painlev\'e type equations, \cite{adler,GGM21b}. It was recently shown that every exceptional polynomial family must be related to a classical family by a sequence of Darboux transformations, \cite{GFGM19}. We first describe here Darboux transformations with seed functions that have no repeated eigenvalues. This is the class of transformations that leads to exceptional orthogonal polynomials of the first kind, when applied on the classical polynomial families, \cite{Bo20,Du15,Du17,GGM14,GKM12}.

In this Section we revise some of these well known results, albeit with a little twist: we describe Darboux transformations for the class of general second order differential operators, which include Schr\"odinger operators as a particular case. For the purpose of this paper, we focus on second order differential operators with rational coefficients that have an infinite number of polynomial eigenfunctions, and we restrict to rational Darboux transformations of these operators that preserve this property by construction. However, the results derived in this section can be trivially extended to general second order operators. All of this section is written at a purely formal level, so in an abuse of notation we speak of differential operators without specifying their domain, or we speak of eigenfunctions as solutions of an eigenvalue problem, without defining a proper spectral theoretic setting. 

\begin{definition}\label{def:ratDT}
  For $n \in \Nz$, let $T_0, T_1,\ldots, T_n$ be second-order rational
  operators. We say that $T_0\to T_1\to \cdots \to T_n$ is an
  \emph{$n$-step rational Darboux transformation} if there exist
  first-order rational operators $A_1, A_2, \ldots, A_n$ such that
  \begin{equation}
    \label{eq:AT=TA} 
    A_{k} T_{k-1} = T_{k} A_{k} ,\quad k=1,\ldots, n.
  \end{equation}
\end{definition}

The next Proposition states that if $A_k$ is a first order intertwiner operator between $T_{k-1}$ and $T_k$ as in \eqref{eq:AT=TA}, then its kernel must be spanned by a formal eigenfunction of $T_{k-1}$.

\begin{prop}
  \label{prop:fchaineigens}
  For $n \in \Nz$, let $T_0 \to T_1\to\dots\to T_n$ be an $n$-step rational Darboux transformation, and $A_1, \ldots, A_n$ be the first-order rational operators satisfying  the
   intertwining relation
  \eqref{eq:AT=TA}. Let $b_k(z)$ and $w_k(z)$  be rational functions such that
    \begin{equation}
        \label{eq:Abwk}
        A_k  = b_k \lp D - w_k \rp,
    \end{equation}
    and define (up to a constant factor) the quasi-rational functions
    \begin{equation}
        \label{eq:psikwk}
        \psi_k(z) := \exp \lp \int^z w_k(u) du \rp.
    \end{equation}
    We then have that
    \begin{equation}
        \label{eq:Tk-1psik}
        T_{k-1} \psi_k = \lambda_k \psi_k,\qquad k=1,\ldots, n,
    \end{equation}
    where $\lambda_1, \lambda_2, \ldots, \lambda_n$ are constants. 
\end{prop}

\begin{proof}
    By \eqref{eq:Abwk} and \eqref{eq:psikwk}we have that $A_k \psi_k=0$, so  \eqref{eq:AT=TA} implies that
    $A_k T_{k-1} \psi_k = 0$.
    Since $A_k$ is a first-order operator, its kernel is
    1-dimensional and it is therefore spanned by $\psi_k$. This implies
    that there exists a constant $\lambda_k$ such that
    \eqref{eq:Tk-1psik} holds. 
\end{proof}

\begin{remark}
  Assume that $T_k$ has the form
    \begin{equation}
        \label{eq:Tpqrk}
        T_k =pD^2+ q_k D + r_k,\quad k=0,1,\ldots, n.
    \end{equation}
    Observe that its coefficients $p,q_k$ and  $r_k$  are related to the rational functions $w_1,\ldots, w_n$ defined in \eqref{eq:Abwk} by the following
    Ricatti-type equations
    \begin{equation}
        \label{eq:pqrwk}
        p (w_k'+w_k^2)+ q_{k-1}w_k + r_{k-1} = \lambda_k,\quad k=1,\ldots, n.
    \end{equation}
\end{remark}

For a given $n$-step Darboux transformation, the
corresponding $\psi_1, \ldots, \psi_n$ are unique, up to a choice of
multiplicative constant.  In light of this remark and future ones, we introduce the following defintions:

\begin{definition}
We call the quasi-rational functions
$\psi_1, \ldots, \psi_n$ \emph{factorization eigenfunctions}, the set of numbers
$\lambda_1,\ldots, \lambda_n$  \emph{factorization eigenvalues}, and the rational functions
 $b_1, \ldots , b_n$ \emph{factorization gauges}.
\end{definition}

A chain of Darboux transformations is more often defined as a factorization of each second-order operator $T_k$ for $k = 0, \ldots, n$, followed by a permutation of the two factors to yield the next operator $T_{k+1}$. We make precise the notion of a factorization chain in the next definition and establish later that both approaches (intertwining of operators and factorization) coincide.

\begin{definition}
  Let $T_0,T_1,\ldots, T_n$ be second-order rational operators.  We
  say that $T_0 \to T_1 \to \cdots\to T_n$ is a \emph{factorization chain} if
  there exist first-order rational operators $A_k, B_k,\; k=1,\ldots,
  n$ and constants $\lambda_1,\ldots, \lambda_n$ such that
    \begin{equation}
        \label{eq:fachain}
        T_{k-1} = \hA_{k} A_{k} + \lambda_{k},\quad T_{k} = A_{k} \hA_{k} + \lambda_{k},\qquad k=1,\ldots, n,
    \end{equation}
\end{definition}

The
intertwining relations given in \eqref{eq:AT=TA} then follow as a
consequence. We now show that these two formulations of Darboux
transformations are equivalent.

\begin{prop}
  \label{prop:fachainequiv}
  A multi-step rational Darboux transformation is necessarily a factorization
  chain, and viceversa.
\end{prop}

\begin{proof}
  Suppose that \eqref{eq:fachain} holds.  Then, the intertwining
  relations \eqref{eq:AT=TA} follow by the associativity of operator
  composition.
  
  Conversely, suppose that \eqref{eq:AT=TA} holds, let  $A_k$ be as in \eqref{eq:Abwk}, and set
    \begin{gather}
        \label{eq:hAkdef}
        \hA_k := \hb_k (D-\hw_k), 
    \end{gather}
    where
    \begin{gather}
        \label{eq:hwkdef}
        \hw_k := - w_k-\frac{q_{k-1}}{p} + \beta_k, \quad \beta_k := \frac{b_k'}{b_k},\quad \hb_k := \frac{p}{b_k}.
    \end{gather}
    Let $\lambda_1,\ldots, \lambda_n$ denote the eigenvalues as per
    \eqref{eq:Tk-1psik}. By \eqref{eq:pqrwk} and a direct calculation,
    we have
    \begin{equation}
        \label{eq:hAkAk}
        \begin{aligned}
            \hA_k A_k &= \hb_k( D -\hw_k) b_k (D- w_k)\\
            &= p (D-\hw_k+\beta_k) (D-w_k)\\
            &= p (D+w_k) (D-w_k)+ q_{k-1}(D-w_k) \\
            &=p D^2 + q_{k-1} D - p (w_k'+w_k^2) - q_{k-1} w_k \\
            &=T_{k-1}-\lambda_k .
        \end{aligned}
    \end{equation}
    Hence, it follows that
    \begin{gather*}
        (T_k - A_k \hA_k-\lambda_k) A_k = A_k (T_{k-1} - \hA_k A_k-\lambda_k)  = 0.
    \end{gather*}
    The ring of differential operators does not have zero
    divisors, so relations \eqref{eq:fachain} follow immediately.
\end{proof}

\begin{remark}
As a direct consequence we may
observe that Darboux transformations are invertible. Indeed, by
\eqref{eq:fachain}, we have that
\begin{gather*}
    \hA_k T_k = T_{k-1} \hA_k,\quad k=1,\ldots, n,
\end{gather*}
where $\hA_1,\ldots,\hA_n$ are the rational operators defined by
\eqref{eq:hAkdef} and \eqref{eq:hwkdef}. Thus, the chain
$T_n \to T_{n-1} \to \cdots \to T_0$ also satisfies the definition of
an $n$-step Darboux transformation.
\end{remark}

\subsection{Seed eigenfunctions and generalized Crum formula} \hfill

So far we have defined $n$-step Darboux transformations referring to the intertwining or factorization of operators $T_0,T_1,\dots,T_n$ at each step of the chain, which requires an eigenfunction $\psi_k$ for each operator $T_{k-1}$ in the chain. In this section we show how to define an $n$-step Darboux transformation using only eigenfunctions of the first operator $T_0$, that we will call \emph{seed functions}. In the case of Darboux chains for Schr\"odinger operators, this construction leads to the well known Darboux-Crum Wronskian formula, \cite{crum}. Theorem~\ref{thm:fachain} in this section can thus be seen as a generalization of Crum's formula.

 Let $T_0$ be a second-order rational operator and let
$\phi_1, \phi_2, \ldots, \phi_{n}$ be quasi-rational eigenfunctions of
$T_0$. Explicitly, we have
\begin{gather*}
    T_0\phi_k = \lambda_k \phi_k,\quad k=1,\ldots, n,
\end{gather*}
where $\lambda_1, \lambda_2, \ldots, \lambda_{n}$ are constants. We
refer to $\phi_1,\ldots, \phi_n$ as \emph{seed eigenfunctions},
because, as we show below, a Darboux transformation is determined by a
choice of seed eigenfunctions and factorization gauges.

Going forward, for a set of indices $\{ i_1, \ldots, i_k \} \subseteq \{ 1,\ldots, n \}$, we write
\begin{equation}
    \label{eq:phiwr}
    \phi_{(i_1, i_2, \ldots, i_k)}
    := \Wr[\phi_{i_1},\phi_{i_2},\ldots,    \phi_{i_k}].
\end{equation}
We now arrive at the key result of this section: explicit formulas for the coefficients of an operator obtained by a multi-step Darboux transformation.  

\begin{thm}
    \label{thm:fachain}
    Let $T_0$ be a second-order rational operator as in \eqref{eq:Tpqrk},
    $\{\phi_1, \ldots, \phi_n\}$ be quasi-rational eigenfunctions of $T_0$
    with distinct eigenvalues $\lambda_1, \ldots, \lambda_n$ and let $\{b_1, \ldots, b_n\}$ be a set of non-zero rational functions. Define the operators
    \begin{equation}
     \label{eq:Tkdef}
    T_k := p D^2 + q_k D + r_k,\qquad        A_k := b_k(D-w_k).
    \end{equation}
where
  \begin{equation}
  \label{eq:omk}
     \sigma_k:=\sum_{j=1}^k (\log b_j)',\quad \upsilon_k := 
        \frac{\phi_{(1, 2, \ldots, k)}'}{\phi_{(1, 2, \ldots, k)}},\qquad
        k=1,\ldots, n; 
  \end{equation}
and
    \begin{align}
        \begin{split}\label{eq:qkrk}
        &q_k := q_0+ k p' - 2 p \sigma_k, \\ 
        &r_k := r_0+k q_0'  + \frac12 k(k-1) p''+ \upsilon_k p'  -  \sigma_k(q_0+kp')  + (\sigma_k^2-\sigma_k' + 2\upsilon_{k}') p, 
        \end{split}\\
      \begin{split}\label{eq:wkgen}
        &w_1  :=\upsilon_1,\quad
        w_k :=
          \sigma_{k-1} +\upsilon_k - \upsilon_{k-1},\quad k=2,\dots,n.
          \end{split}
    \end{align}
Then, the sequence of operators  $T_0 \to T_1 \to \cdots \to T_n$ is an $n$-step Darboux
    transformation in the sense of Definition~\ref{def:ratDT}, i.e. \eqref{eq:AT=TA} holds with the operators $A_k$ defined above.
\end{thm}

Since we assume that the eigenvalues $\lambda_1, \ldots, \lambda_n$ are distinct, $\phi_1, \ldots, \phi_n$ are linearly independent. Hence, the Wronskians in  the denominator of \eqref{eq:omk} are non-zero, and the functions $\upsilon_k$ are well-defined.

\begin{remark}
The well-known Crum formula for an $n$-step Darboux transformation of a
  Schr\"odinger operator
\[ T_n=T_0 -2 (\log \phi_{(1, 2, \ldots, n)})'' \]  
 is a special case of Theorem ~\ref{thm:fachain} that corresponds to starting from a Schr\"odinger operator ($p=-1$ and $q_0=0$) and choosing the gauge $b_1=\ldots = b_n = 1$. Thus, this result can be regarded as an extension of the Darboux-Crum formula from Schr\"odinger operators to general second order operators. Although for the purpose of the paper we restrict to second order operators with rational coefficients and polynomial eigenfunctions, the generalized Crum formula in Theorem ~\ref{thm:fachain} is valid on a general setting.
\end{remark}

The rest of this section is devoted to the proof of Theorem ~\ref{thm:fachain}. The strategy is to first establish the result for the particular gauge $b_1=\cdots = b_n = 1$ and then show the transformation rules under a different gauge. Most of the proofs proceed by induction and make use of differential algebra and properties of Wronskian determinants. We will first need to state and prove a number of auxiliary lemmas, where operators with tilde denote the restriction of the same objects in  Theorem~\ref{thm:fachain} to the special case $b_1=\cdots = b_n = 1$.

\begin{lem}
    \label{lem:phiij}
   Under the same setting as in Theorem ~\ref{thm:fachain}, let  $\tA_k = D- \tw_k$ for $k\in\{1,\dots, n\}$ with $\tw_k= \upsilon_k-\upsilon_{k-1}$. Then,
    \begin{gather}
        \label{eq:phiij} 
        (\tA_{k-1} \tA_{k-2} \cdots \tA_1) \phi_j= \frac{\phi_{(1, 2, \ldots,
            k-1, j)} }{\phi_{(1, 2, 
            \ldots, k-1)}} ,\quad 2\le k \le j \le n.
    \end{gather}
\end{lem}

\begin{proof}
  For convenience, set
    \begin{align}
        \label{eq:phiijdef}
        &\varphi_{k\tor{,}j} := (\tA_{k-1} \tA_{k-2} \cdots \tA_1) \phi_j.
    \end{align}
  We argue by induction. The case of $k=2$ follows directly from the definition. Indeed,
    \begin{gather*}
      \varphi_{2,j} = \tA_1 \phi_j = \frac{1}{\phi_1}\Wr[\phi_1,\phi_j] =
      \frac{\phi_{1,j}}{\phi_1}.
    \end{gather*}
    Now, we suppose that \eqref{eq:phiij} holds for a
    particular $k<n$. By the inductive hypothesis, in particular we have
    \[
        \varphi_{k,k} = \frac{\phi_{(1, 2, \ldots, k)}}{\phi_{(1, 2,
          \ldots, k-1)}},\quad
        \tilde\omega_k=\frac{\varphi_{k,k}'}{\varphi_{k,k}} \quad \mbox{ and } \quad
      \tA_k = D - \frac{\varphi_{k,k}'}{\varphi_{k,k}}.
      \]
    Hence,
    \begin{align*}
      \varphi_{k+1,j}
      &= (\tA_k \cdots \tA_1) \phi_j = \tA_k \varphi_{k,j} \\
      &= \frac{1}{\varphi_{k,k}} \Wr \lb \varphi_{k,k}, \varphi_{k,j} \rb 
      = \frac{1}{\varphi_{k,k}} \Wr\lb \frac{ \phi_{(1, 2, \ldots, k-1,
        k)}}{\phi_{(1, 2, \ldots, k-1)}}, \frac{ 
        \phi_{(1, 2, \ldots, k-1, j)}}{\phi_{(1, 2, \ldots, k-1)}} \rb.
    \end{align*}
    By well-known properties of the Wronskian operator, we finally have
    \[ \varphi_{k+1,j} =\frac{1}{\varphi_{k,k}} \frac{ \Wr[\phi_{1, 2, \ldots, k-1,k},
        \phi_{(1, 2, \ldots, k-1,j)}]}{ \lp \phi_{(1, 2,
          \ldots, 
        k-1)} \rp^{2}} = \frac{ \phi_{(1, 2, \ldots, k-1)} \phi_{(1, 2, \ldots, k,
        j)}}{\varphi_{kk} \lp \phi_{(1, 2, \ldots,
        k-1)} \rp^{2}} = \frac{\phi_{(1, 2, \ldots, k, j)}}{\phi_{(1, 2, \ldots,
          k)}}.\]
\end{proof}

The following lemma establishes the desired result in the restricted gauge $b_1=\cdots = b_n = 1$. 

 \begin{lem}\label{lem:tTkint}
   Under the same setting as in Theorem ~\ref{thm:fachain}, let $\tT_0 = T_0$ and
   $\tT_k := p D^2 + \tq_k D + \tr_k$ for $k\in\{1,\ldots n\}$, where
   \begin{align}
     \label{eq:tqk}
     \tq_k
     &:= q_0 + k p'\\
     \label{eq:trk}
     \tr_k :&= r_0+k q_0'  + \frac12 k(k-1) p''+ \upsilon_k p'   + 
            2\upsilon_{k}' p.
   \end{align}
   Let $\tA_k$ be as in Lemma ~\ref{lem:phiij}.
   Then
   \begin{equation}
     \label{eq:tAktTk}
     \tA_k \tT_{k-1} = \tT_k \tA_k,\quad k=1,\ldots, n.
   \end{equation}
 \end{lem}

\begin{proof}
 By direct
  calculation, we find that
    \begin{align*}
      \tT_k \tA_k &= (pD^2 + \tq_k D+ \tr_k)(D -  \tw_k)\\
              &=   p D^3+ (\tq_k-p\tw_k) D^2 +(\tr_k-\tq_k \tw_k-2 p \tw_k') D  
                - (\tr_k\tw_k +\tq_k\tw_k' + p \tw_k'')\\
      \tA_k \tT_{k-1} &= (D -  \tw_k)(pD^2 + \tq_{k-1} D+ \tr_{k-1})\\
              &=  p D^3+ (\tq_{k-1}-p\tw_k+p') D^2 + \big(\tr_{k-1} -
                \tq_{k-1}\tw_k + \tq_{k-1}' \big)D+ \tr_{k-1}'- \tw_k \tr_{k-1} 
    \end{align*}
    Hence, by inspection of the coefficients, the desired intertwining
    relation is equivalent to the following three relations:
    \begin{align}
        \label{eq:qkqk-1}
      &\tq_k      = \tq_{k-1}+ p'\\
      \label{eq:rkrk-1}
      &\tr_k = \tr_{k-1} +  \tq_{k-1}' + \tw_{k} p'  +  2 p\tw_{k}' \\
      \label{eq:trktwk}
      &\tr_k\tw_k +\tq_k\tw_k' + p \tw_k'' =  \tr_{k-1}\tw_k - \tr_{k-1}'
    \end{align}
    By inspection, \eqref{eq:tqk} entails \eqref{eq:qkqk-1}.  Then,
    using \eqref{eq:tqk}, \eqref{eq:trk} and $\tilde w_k=\upsilon_k-\upsilon_{k-1}$, we find that
    \begin{align*}
      \tr_k - \tr_{k-1} - \tq_{k-1}'
      & = q_0' + (k-1) p'' + (\upsilon_k - \upsilon_{k-1}) p' + 2(
        \upsilon_{k}' - \upsilon_{k-1}') p  -(q'_0 + (k-1)p'' )\\ 
      & = \tw_k p' + 2 \tw_k'p,
    \end{align*}
    which establishes \eqref{eq:rkrk-1}. Using \eqref{eq:qkqk-1} and \eqref{eq:rkrk-1}, we can rewrite
    \eqref{eq:trktwk} as
    \begin{align*}
      0&=(\tr_k-\tr_{k-1})\tw_k +\tq_k\tw_k' + p \tw_k'' + \tr_{k-1}'\\
       &=(\tq_{k-1}' + p'  \tw_{k}+2p \tw_k')\tw_k    +(\tq_{k-1}+p')\tw_k'
         + p \tw_k''   +\tr_{k-1}'\\ 
       & =\big( p(\tw_k' + \tw_k^2)+ \tq_{k-1} \tw_k + \tr_{k-1}\big)'
    \end{align*}
    Let $\tpsi_k$ denote
    \[\tpsi_k =\varphi_{k,k}= \frac{\phi_{(1,\ldots,k)}}{\phi_{(1,\ldots, k-1)}},\quad \tw_k = (\log\tpsi_k)',\]
    and observe that
    \[ p(\tw_k' + \tw_k^2)+ \tq_{k-1} \tw_k + \tr_{k-1} = \frac{\tT_{k-1}\tilde \psi_k}{\tilde \psi_k},\]
    so \eqref{eq:trktwk} is equivalent to establishing  
    \[\tT_{k-1} \tpsi_k = \lambda_k \tpsi_k,\qquad k \in\{1,\dots, n\}.\] 
    
    The rest of the proof follows by induction.  The base case holds because
    $\tpsi_1 = \phi_1$ and by assumption,
    $T_0 \phi_1= \lambda_1 \phi_1$. Now, suppose that we have established
    \eqref{eq:tAktTk} for $j = 1,\ldots, k-1$. By Lemma
    ~\ref{lem:phiij} and the inductive hypothesis,
    \begin{align*}
      \tT_{k-1} \tpsi_k &= \tT_{k-1} \tA_{k-1} \cdots \tA_1 \phi_k 
        = \tA_{k-1}  \tT_{k-2} \tA_{k-2} \cdots \tA_1 \phi_k \\
        &= \tA_{k-1} \cdots \tA_1 T_0 \phi_k 
        =\lambda_k \tA_{k-1} \cdots \tA_1  \phi_k
        = \lambda_k \tpsi_k.
    \end{align*}
\end{proof}



The next lemma shows the gauge transformation that connects $T_k$ with $\tT_k$.
\begin{lem}
  \label{lem:gaugexform}
  Let $T_k$ be as in Theorem ~\ref{thm:fachain} and $\tilde T_k$ be as in Lemma ~\ref{lem:tTkint}. Setting
  $s_k := b_1\dots b_k$, we have
 \begin{equation}
   \label{eq:gaugexform}
   T_k = s_k \tT_k s^{-1}_k,\quad k=1,\ldots, n.
 \end{equation}
\end{lem}
\begin{proof}
  Observe that $\sigma_k = (\log s_k)'$. By direct calculation,
  \begin{align*}
    s_k D s_k^{-1}
    &= D-\sigma_k,\\
    s_k D^2 s_k^{-1}
    &= (s_k D s_k^{-1})^2
    = (D-\sigma_k)^2 = D^2 - 2\sigma_k D + \sigma_k^2 - \sigma_k',
  \end{align*}
  Hence,
  \begin{align*}
    q_k &=  \tq_k - 2\sigma_k p,\\
    r_k &= \tr_k +(\sigma_k^2-\sigma_k')p -  \sigma_k \tq_k,
  \end{align*}
  as was to be shown.
\end{proof}

With all the previous elements, proving Theorem ~\ref{thm:fachain} for general factorization gauges becomes a straightforward computation

\begin{proof}[Proof of Theorem ~\ref{thm:fachain}]
  We observe that $A_k$ is related with $\tilde A_k$ in Lemma ~\ref{lem:phiij} by $s_k=b_1\dots b_k$ as follows:
  \[ s_k \tA_k s_k^{-1} = D- \tw_k - \sigma_k = A_k.\]
  Hence, by Lemma ~\ref{lem:gaugexform},
  \[ T_k A_k  = s_k \tT_k \tA_k s_k^{-1} = s_k \tA_k
    \tT_{k-1} s_{k}^{-1} = A_k T_{k-1},\]
  as was to be shown.
\end{proof} 

\begin{remark}
  In light of Lemma ~\ref{lem:gaugexform}, a different choice of
  $b_1,\ldots, b_n$ results in a gauge transformation of the operators
  $T_1,\ldots, T_n$.  It is for this reason that we refer to
  $b_1,\ldots, b_n$ as factorization gauges.  Moreover, by
  \eqref{eq:qkrk}, the coefficients of $T_n$ are defined directly in
  terms of $\sigma_n= (\log s_n)'$ and $\upsilon_n$.  The Wronskian
  $\phi_{(1, 2, \ldots, n)}$ is alternating in its indices and its
  log-derivative $\upsilon_n$ is invariant with respect to
  permutations of the set $\{ 1,2, \ldots, n \}$.  Thus, we see that
  $T_n$ depends only on the choice of the seed eigenfunctions
  $\phi_1,\ldots, \phi_n$ ---  irrespective of their order --- and on
  the product of the factorization gauges $s_n= b_1 \cdots b_n$.
  Fixing the seed eigenfunctions, but choosing a different $s_n$
  amounts to a gauge transformation of $T_n$.
\end{remark}

\section{Confluent Darboux transformations} 
\label{sec:CDT}

In this section we generalize the concept of Darboux transformations  to allow
for repeated eigenvalues. Notice that the  construction in Section ~\ref{sec:algDar} fails
if the eigenvalues of the factorization eigenfunctions are not all distinct, because then some of the seed
functions may not be linearly independent, which leads to the
vanishing of the Wronskians in the denominator of \eqref{eq:omk}. In order to allow for
repeated eigenvalues, we will allow some of our seed eigenfunctions to
become generalized eigenfunctions.

\begin{definition}
  We say that two second-order rational operators $T_0$ and  $T_2$ are connected by a \emph{confluent Darboux transformation} if there exists a second-order
  rational operator $T_1$ such that $T_0 \to T_1 \to T_2$ is a 2-step
  Darboux transformation and the corresponding eigenvalues, as
  defined in Proposition~\ref{prop:fchaineigens}, satisfy
  $\lambda_1 = \lambda_2$.
\end{definition}


As we shall see,  the factorization eigenfunction for the second step $T_1\to T_2$ will not be related to an eigenfunction of $T_0$ but to a generalized eigenfunction, which motivates  the following definition.

\begin{definition}
  Let $T$ be a linear differential operator. We say that $\phi$ is an
  \emph{$n\supth$ order generalized eigenfunction} of $T$ if
  $(T - \lambda)^n \phi= 0$, but $(T-\lambda)^{n-1}\phi \ne 0$.
\end{definition}

For our purposes it will be sufficient to use only
first-order generalized eigenfunctions; however, the more general
construction can be found in \cite{CASH17}.  We now show that a CDT
(confluent Darboux transformation) can be generated by a seed
eigenfunction and a corresponding second-order generalized
eigenfunction.

In order to build a CDT, we start with a second-order rational operator
$T_0 = p D^2 + q_0 D + r_0$, a quasi-rational seed eigenfunction
$\phi$ with eigenvalue $\lambda$ and rational factorization gauges
$b_1$ and $b_2$.  Let $T_0 \to T_1$ be a 1-step Darboux
transformation with factorization function $\phi$ at factorization eigenvalue $\lambda$. We now wish to perform a second Darboux transformation on $T_1$ using the repeated
eigenvalue $\lambda$ and factorization gauge $b_2$. As shown below,
this requires that the second seed function be a generalized
eigenfunction of $T_0$.  We therefore seek to construct a function
$\phi^{(1)}$ such that $(T_0 - \lambda) \phi^{(1)}= \phi$. The following lemma shows how to achieve this.
\begin{lem}\label{lem:geneigen}
  Let $T_0=pD^2+q_0D+r_0$ and let $\phi$ an eigenfunction of $T_0$ with eigenvalue $\lambda$, i.e. $T_0\phi=\lambda\phi$.
  Then the  particular solution of the inhomogeneous equation $(T_0 - \lambda) y= \phi$ is given by
  \begin{equation}
    \begin{aligned}
        \label{eq:phi2int0}
        \phi^{(1)}(z) &=  \phi(z) \int^z \lp\frac{
           \mu(u) }{ \phi(u)^2 }\int^u\frac{\phi^2(s)}{p(s)\mu(s)}ds\rp du, 
    \end{aligned}
\end{equation}
where 
\begin{equation}
  \label{eq:mudef}
  \mu(z) :=
  \exp\lp - \int^z \frac{q_0(u)}{p(u)} du \rp, 
\end{equation}
Also, a linearly independent solution of the homogenous equation $(T_0-\lambda)y=0$ is given by
\begin{equation}
       \phiperp(z):= \phi(z) \int^z \frac{\mu(u)}{\phi(u)^2} du.
\end{equation}
\end{lem}

\begin{remark}
  Having fixed $\mu$ as in \eqref{eq:mudef}, the above definition of $\phi^{(1)}$ incorporates two additional constants of integration which correspond to linear combination of the solutions of the homogeneous equation, $\phi$ and $\phi^\perp$. We will fix one of these constants once we consider an explicit form for $p$, by imposing an appropriate lower bound for the integral. The other constant of integration serves as a natural deformation parameter in the CDT construction. This will be explained in more detail later on.

\end{remark}

\begin{proof}[Proof of Lemma ~\ref{lem:geneigen}]
We first consider a
complementary solution to the eigenvalue equation $T_0 y = \lambda y$,
which we will denote $\phiperp$. We can obtain an explicit formula
for $\phiperp$ via reduction of order. Substituting the form
$\phiperp := f \phi$ into the equation
$T \phiperp = \lambda \phiperp$ yields the equation
\begin{gather*}
    \frac{f''}{f'}  = - \frac{q_0}{p}  - 2 w,\quad \mbox{ with } w= \frac{\phi'}{\phi}.
\end{gather*}
Noticing that  $\frac{\mu'}{\mu} = -\frac{q_0}{p}$, it follows that
\begin{align*}
    f(z) = \int^z \frac{\mu(u)}{\phi(u)^2} du.
\end{align*}
The above definitions are purely formal in that we have not specified
the lower bound of the integrals. This means that $\phiperp$ is
defined up to a constant multiple of $\phi$, and that $\mu$ is defined
up to a choice of positive multiplicative constant.

We now construct $\phi^{(1)}$ 
using variation of parameters. We set
\begin{equation}
    \label{eq:phi2def}
    \phi^{(1)} := \hrho \phi + \rho \phiperp,
\end{equation}
where $\hrho$ and $\rho$ are unknown functions satisfying
\begin{gather*}
    \hrho{}' \phi + \rho' \phiperp = 0.
\end{gather*}
Since $ (T_0 - \lambda)  \phi^{(1)}  = \phi$, then they also satisfy
\begin{align*}
    p\big(\hrho''\phi+2\hrho'\phi'+\rho''\phiperp+2 \rho'(\phiperp)'\big)=\hrho \phi + \rho \phiperp.
\end{align*}

Solving for the functions $\hrho$ and $\rho$ satisfying the above
system of equations, we find that
\begin{align}\nonumber
      \hrho(z) &= - \int^z \phi(u) \phi^{\perp}(u) W(u) du , \\\label{eq:rhodef}
      \rho(z) &= \int^z \phi^2(u) W(u) du ,
\end{align}

where
\begin{equation}
    \label{eq:Wpmu}
    W(z) := \big(p(z)\mu(z)\big)^{-1} = p(z)^{-1} \exp \lp \int^z
    \frac{q_0(u)}{p(u)} du \rp. 
\end{equation}
Now integrating by parts it follows that we can express $\phi^{(1)}$ as 
\begin{equation}
    \begin{aligned}
        \label{eq:phi2int}
        \phi^{(1)}(z) &= \phi(z) (f(z)\rho(z)+\hrho(z)) =
        \phi(z) \lp f(z) \rho(z) - \int^z f(u) \rho'(u) du \rp \\  
        &= \phi(z) \int^z \rho(u) f'(u) du = \phi(z) \int^z \frac{
          \rho(u) \mu(u) }{ \phi(u)^2 } du.
    \end{aligned}
\end{equation}
\end{proof}

After the first Darboux transformation $T_0\to T_1$ at factorization eigenvalue $\lambda$, we define the intertwiner $A_1 = b_1(D- w_1)$, with $w_1=(\log\phi)'$. The first candidate for factorization function for the second Darboux transformation $T_1\to T_2$ would be the image of $\phiperp$ under $A_1$:
\begin{equation}
    \label{eq:psi1perp}
    \psi^{\perp} := A_1 \phiperp = b_1\frac{\mu}{\phi}.
\end{equation}

Indeed, since $T_1 \psi^{\perp} = \lambda \psi^{\perp}$, we could employ $\psi^{\perp}$ as a factorization eigenfunction for a 1-step Darboux transformation on $T_1$. However, this choice of eigenfunction produces the
\emph{inverse Darboux transformation} $T_1 \to T_0$.

In order to construct an operator $T_2$ distinct from $T_0$, we need to consider another candidate: the image of the generalized eigenfunction $\phi^{(1)}$ under $A_1$. Indeed, we define the factorization eigenfunction for $T_1$ to be
\begin{equation}
  \label{eq:psi2def}
  \psi_2 := A_1 \phi^{(1)} = \rho \psiperp,
\end{equation}
with $\rho$ as in \eqref{eq:rhodef}. The second equality is true because
\[
    \Wr[\phi,\phi^{(1)}] = \Wr \lb \phi, \phi \int^z \frac{\rho \mu}{ \phi^{2}}
    \rb = \rho \mu \]
and therefore
\[
    \psi_2 = \frac{b_1}{\phi} \Wr[\phi,\phi^{(1)}] = \frac{b_1}{\phi}
    \rho\mu = \rho \psiperp.
\]
 The key observation is that although $\phi^{(1)}$ is only a generalized eigenfunction of $T_0$, its image $\psi_2$ is a true eigenfunction of $T_1$ at eigenvalue $lambda$:
\begin{equation}
    \label{eq:T1psi2}
    (T_1-\lambda)\psi_2= (T_1-\lambda) A_1 \phi^{(1)} = A_1 (T_0-\lambda) \phi^{(1)} = A_1 \phi_1 = 0,
\end{equation}
and thus it can be employed for the second Darboux transformation $T_1\to T_2$.
We summarize the construction of a confluent Darboux transformation $T_0\to T_2$ with a generalized eigenfunction in the following proposition.

\begin{prop}
  Let $T_0=pD^2+q_0D+r_0$ be a second-order rational operator, $\phi$ a
  quasi-rational eigenfunction of $T_0$ with eigenvalue $\lambda$, and
  $b_1, b_2$ a choice of non-zero rational functions.  Let $\mu,\rho$
  be defined as per \eqref{eq:mudef}, \eqref{eq:rhodef}, respectively,  and assume
  that $\rho$ is a quasi-rational function. Let $T_2=pD^2+q_2D+r_2$, where $p_2$ and $r_2$ are defined in \eqref{eq:qkrk} with
  \[ \upsilon_1 := (\log\phi)',\quad \upsilon_2 := \big(\log(\rho\mu)\big)'.\] 
Then
  $T_0 $ and $ T_2$ are connected by a confluent Darboux transformation
\end{prop}

\begin{proof}
  Firstly, we observe that $T_2$ is a second-order rational operator because of the assumption on $\mu$.
  Then, it remains to show there is $T_1$ such that $T_0\to T_1\to T_2$ is a 2-step DT with $\lambda_1=\lambda_2$.
  Defining $T_1, A_1, A_2$ 
  as per \eqref{eq:Tkdef}-\eqref{eq:wkgen}, with $b_j, \upsilon_j$, $j\in\{1,2\}$ as in the statement, the result then follows.  
\end{proof}

\begin{remark}
  Observe that the construction of a confluent Darboux transformation amounts to applying the extended Crum Wronskian formula derived in Theorem~\ref{thm:fachain} for $k=2$ with the only modification that one of the seed functions of $T_0$ is a true eigenfunction, but the second one is a generalized eigenfunction, i.e. apply the formulas in Theorem~\ref{thm:fachain} with
\[ \phi_1=\phi, \quad \phi_2=\phi^{(1)},\quad \lambda_1=\lambda_2=\lambda.\]
\end{remark}

\begin{remark}
  The definition \eqref{eq:phi2int0} of the second seed function $\phi^{(1)}$ involves two indefinite integrals. The outer integral provides no extra freedom, because it means that $\phi^{(1)}$ is defined up to an additive term $C\phi$ but this term will vanish in the Wronskian $\phi_{(1,2)}$. The lower bound of the inner integral, however, introduces a free real parameter, which is a characteristic feature of the CDT.
\end{remark}

\begin{remark}
  The assumption that $\frac{\rho'}{\rho}$ be rational ensures that
  $\upsilon_2$, as defined above, is rational.  This assumption may be restated as the condition that $\phi^2 W$, the integrand of
  \eqref{eq:rhodef}, is a rational function with vanishing residues. Verifying the rationality of the CDT is key to ensure that the transformed operator has polynomial eigenfunctions.
\end{remark}

\begin{remark}
  The confluent aspect of a CDT comes from a conceptual formula for the generalized eigenvalue equation $(T_0 - \lambda) \phi^{(1)} =\phi$. We note that, despite not being made explicit, the seed
eigenfunction $\phi$ depends on the eigenvalue $\lambda$. This
dependence can be recovered by imposing initial conditions on $\phi$
and $\phi'$. Thus, starting from the eigenvalue equation
$(T_0 - \lambda)\phi = 0$, we can differentiate with respect to
$\lambda$ to find that
\begin{gather}
    (T_0 - \lambda) \lb \frac{\partial \phi}{\partial \lambda} \rb = \phi .
\end{gather}
This equation implies a rather simple formula for defining $\phi^{(1)}$,
which is
\begin{gather}
    \phi^{(1)} = \frac{\partial \phi}{\partial \lambda}.
\end{gather}
This formula is not of much practical use, since the functional
dependence on $\lambda$ is typically impossible to express
explicitly. However, this expression is of conceptual importance and
the derivative can be seen as a limiting case of the ordinary Darboux
transformation, where the eigenvalues converge as
$\lambda_2 \to \lambda_1$. Hence the name \emph{confluent} Darboux
transformation.
\end{remark}

In this section we have seen how to build a confluent Darboux transformation as a 2-step  Darboux transformation with a generalized eigenfunction, introducing in the process a free real parameter.
This construction can be iterated at different eigenvalues to create chains of CDTs: perform a CDT
on the first operator $T_0$ at eigenvalue $\lambda_1$, which is
followed by a CDT on $T_2$ at eigenvalue $\lambda_2 \neq \lambda_1$, which yields the operator $T_4$, etc. 

Chains of operators may be constructed through an arbitrary finite number of CDTs in this
fashion, thereby leading to an operator with an arbitrary number of free real parameters, which is Darboux connected to the original $T_0$. In the following section we show how to construct a CDT chain starting on the Gegenbauer polynomials, and leading to the exceptional Gegenbauer polynomials of the second kind.

\section{Exceptional Gegenbauer operators and polynomials}
\label{sec:Gegen}

In this Section we apply the theory developed in Sections~\ref{sec:algDar} and \ref{sec:CDT} to construct a chain of confluent Darboux transformations on the classical Gegenbauer operator. 

\subsection{Definition of exceptional Gegenbauer operators and polynomials} 
In \cite{GFGM19} it was shown that an exceptional operator in the Hermite, Laguerre or Jacobi class must have a very specific form. We define in this section exceptional Gegenbauer polynomials and operators attending to this particular form (as a particular class of exceptional Jacobi operators), postponing for later sections the discussion on how the construction of specific families is achieved. 

\begin{definition}\label{def:GegOp}
  Let $\tau=\tau(z)$ be a non-zero polynomial and $\alpha\in \mathbb{R}$. We say that the differential expression
    \begin{equation}
        \label{eq:GegOp}
        \Tal_\tau(z,D_z) := (1-z^2) \lp D_z^2 - 2 \frac{\tau_z}{\tau} D_z +
        \frac{\tau_{zz}}{\tau} \rp - (2\alpha + 1) z D_z + (2\alpha - 1)
        z \frac{\tau_z}{\tau} .  
    \end{equation}
    is an \emph{exceptional Gegenbauer operator} if $\Tal_\tau$ admits
    eigenpolynomials $\{ \pi_i(z) \}_{i \in \Nz}$ such that the degree
    sequence $\{ \deg \pi_i\}_{i\in\Nz}$ is missing finitely many
    ``exceptional'' degrees.
\end{definition}
In the above definition, it should be stressed that only very specific
polynomials $\tau(z)$ in \eqref{eq:GegOp} will lead to $\Tal_\tau$
being an exceptional Gegenbauer operator, i.e. having an infinite
number of polynomial eigenfunctions. The following sections are
devoted to describing a class of polynomials $\tau$ obtained by
applying a multi-step CDT on $\tau=1$, which ensures that this is
indeed the case.
Note that there is no restriction on the
parameter $\alpha$ at this stage. In the following section we will see
that $\alpha$ must be a half-integer for the Confluent Darboux
transformations to be rational. This means that for standard Darboux
transformations (see Section~\ref{sec:algDar}) the parameter $\alpha$
can be real, leading to $X$-Gegenbauer polynomials of the first kind,
but for confluent Darboux transformations (see Section~\ref{sec:CDT})
the parameter $\alpha\in\Nz+1/2$, leading to excecptional Gegenbauer
polynomials of the second kind.

Observe also, by contrast to classical orthogonal polynomials, that we \emph{are not}
assuming that $\deg \pi_i =i$.  Furthermore, without loss of
generality, it will be convenient to assume that 
\begin{align*}
    \deg \pi_i\neq\deg\pi_j\;  \mbox{ if } i\neq j.
\end{align*}

As usual, we speak of exceptional Gegenbauer polynomials when the
eigenpolynomials $\{ \pi_i(z) \}_{i \in \Nz}$ define a complete
orthogonal polynomial system.

\begin{definition}
    \label{def:GegPol}
    Let $\tau(z)$ be a polynomial and $\alpha\in \mathbb{R}$ .  We say that the set
    $\{\pi_i(z)\}_{i\in\Nz}$ is a family of \emph{exceptional
      Gegenbauer polynomials} with weight
    \begin{equation}
        \label{eq:Wtaldef}
        \Wal_{\tau}(z) :=  \frac{(1-z^2)^{\al-\frac12}}{\tau(z)^{2}} ,
      \end{equation}
    if the following conditions hold:
    \begin{enumerate}
    \item $\tau(z)$ does not vanish on $I=[-1,1]$;
    \item $\{\pi_i(z)\}_{i\in\Nz}$ are eigenpolynomials of an $X$-Gegenbauer operator \eqref{eq:GegOp};
    \item
      The polynomials $\{\pi_i(z)\}_{i\in\Nz}$ form a complete set
      in the Hilbert space $\rL^2 (I,\Wal_\tau)$.
    \end{enumerate}
\end{definition}

Note that there is no need to include an explicit orthogonality
assumption in the above definition, because orthogonality of the
eigenpolynomials follows from assumptions (a) and (b).  Also note that
there is no need for supplementary assumptions regarding the
corresponding eigenvalues as these are necessarily
 quadratic functions of the degree sequence. More specifically, we establish these results in the next two lemmas.
 
 \begin{lem} 
  \label{prop:deglamdai}
 Let  $\Tal_\tau$ be an exceptional Gegenbauer operator and let $\{\pi_i\}_{i\in\Nz}$ be its associated eigenpolynomials,  i.e.
  \begin{equation}
    \label{eq:Ttaupi}
    \Tal_{\tau} \pi_i = \lambda_i \pi_i
  \end{equation}
    Then,  necessarily 
\begin{align}
  \label{eq:lambdaform}
  \lambda_i
  &= -d_i(2\alpha+d_i),\quad \text{where }
  d_i
  =    \deg    \pi_i-\deg\tau 
\end{align}
\end{lem}

\begin{proof}
  Let $\tau(z)$ be a polynomial and $\pi_i(z)$ an eigenpolynomial of
  $\Tal_\tau$ with eigenvalue $\lambda_i$.
  Expliciltly, by \eqref{eq:GegOp}, we have
  \begin{equation}
    \label{eq:Ttaupii}
    (1-z^2) \lp \tau \pi_i'' - 2 \tau' \pi_i' + \tau'' \pi_i \rp -
    (2\alpha + 1) z \tau \pi_i' + (2\alpha - 1) z \tau' \pi_i =
    \lambda_i \tau \pi_i .
  \end{equation}
  Let $n=\deg\pi_i$ and $m=\deg \tau$ and without loss of generality
  suppose that both $\tau(z), \pi_i(z)$ are monic; i.e.
  $\tau(z) = z^m + \ldots$ and $\pi(z) = z^n + \ldots$.  Notice that
  the highest power of $z$ on either side of \eqref{eq:Ttaupii} is
  $m+n$. Hence, for the above equation to hold, the two coefficients
  on $z^{m+n}$ must be equal. Considering only the highest power of
  $z$ in each term, yields the equation
    \begin{gather*}
        \lambda_i z^{m+n} = \lp - n (n-1) + 2 m n - m (m-1) - (2\al+1) n + (2\al-1) m \rp z^{m+n} .
    \end{gather*}
    Hence
    \begin{align*}
        \lambda_i &= - n (n-1) + 2 m n - m (m-1) - (2\al+1) n + (2\al-1) m \\
        &= -n^2 + 2mn - m^2 - 2 \al n + 2 \al m \\
        &= - (n-m)^2 - 2 \al (n-m) \\
        &= - (n-m) ( 2\al + n-m ) ,
    \end{align*}
    as was to be shown.
\end{proof}

\begin{lem}
  \label{prop:orth}
  A family of exceptional Gegenbauer polynomials $\{\pi_i\}_{i\in\Nz}$  is necessarily
  orthogonal with respect to the corresponding weight \eqref{eq:Wtaldef}:
  \begin{equation}\label{eq:orth}
    \int_{I} \pi_{i}(z)\pi_{j}(z) \Wal_\tau(z) dz = 0\;\mbox{ if } i\neq j.
  \end{equation}
\end{lem}

\begin{proof}
  Multiplying the
  eigenvalue equation $\Tal_\tau y = \lambda y$ by $\Wal_\tau$ yields
  a Sturm-Liouville eigenvalue equation
    \begin{equation}
      \label{eq:SLform}
       (\Pal_\tau  y')' + \Ral_\tau y =
       \Wal_\tau\lambda y,
     \end{equation}
     where
     \[
       \begin{aligned}
         \Pal&:= (1-z^2)^{\al+\frac12} \tau^{-2}\\
         \Ral&:=(1-z^2)^{\al+\frac12} \tau_{zz}\tau^{-3} + (2\al-1)
         z(1-z^2)^{\al-\frac12} 
         \tau_z\tau^{-3}
       \end{aligned} \]  Lagrange's identity now gives
     \begin{equation}
       \label{eq:lagid}
       \int \lp y_1 \Tal_\tau y_2 -
       y_2 \Tal_\tau y_1\rp  \Wal_\tau = \Pal_\tau \Wr[y_1,y_2].
     \end{equation}
      Since the eigenvalues $\lambda_i$ are distinct,
     assumption (a) and \eqref{eq:lagid} imply that the
     eigenpolynomials satisfiy orthogonality relations.
\end{proof}

The base case of the class of exceptional Gegenbauer operators is the
classical Gegenbauer operator $\Tal= \Tal_{\tau_0}$, where
$\tau_0(z)=1$.  In this case the general form \eqref{eq:GegOp}
simplifies to \eqref{eq:T0def} and    the
eigenpolynomials are the classical Gegenbauer polynomials \cite{AS64}
\begin{equation}
  \label{eq:gegpoly}
     \Cal_{i} := \sum_{k=0}^{\lfloor i/2 \rfloor} (-1)^k \frac{\Gamma(i-k+\al)}{\Gamma(\al) k! (i-2k)!} (2z)^{i-2k} .
\end{equation}
These classical orthogonal polynomials do have $\deg \Cal_i =i$, and they satisfy the eigenvalue relation
\begin{gather*}
  \Tal \Cal_i = \lambda_i \Cal_i,\quad i \in \Nz ,
\end{gather*}
with $\lambda_i=-i(2\alpha+i)$.  If $\al>-\frac 12$ the polynomials $\{\Cal_i\}_{i\in\Nz}$ form  
a complete set in $\rL^2(I, W^{(\alpha)})$ and they satisfy the orthogonality relation
\[   \int_{I} \Cal_{i}(u) \Cal_{j}(u) \Wal(u) du =
  \nual_{i} \delta_{ij}, \quad i,j \in \Nz \]
where
\begin{align}
  \label{eq:Wdef}
  \Wal(z) &:= \Wal_{\tau_0}(z)=(1-z^2)^{\al-\frac12};\\
  \label{eq:nuidef}
  \nual_i &:= \frac{\pi 2^{1-2\al} \Gamma(i+2\al)}{i!
    (i+\al) \Gamma(\al)^2} ,\quad i\in \Nz.
\end{align}

We can now apply the theory of CDTs to generate exceptional orthogonal
polynomials.

Once the class of exceptional Gegenbauer operators and polynomials has been defined, we will show that this class is not empty by describing in the next sections the construction of exceptional Gegenbauer polynomials of the second kind via a sequence of confluent Darboux transforamtions.

\subsection{Factorizations of exceptional Gegenbauer operators} \hfill

Let us start by describing the factorization of an $X$-Gegenbauer operator \eqref{eq:GegOp} in order to define a $1$-step Darboux transformation. After that, we will combine two Darboux transformations to define a CDT. 

Given rational functions
$\tau(z)$ and $\pi(z)$, and a real constant $\alpha$, we define the
following two first-order rational operators:
\begin{equation}
    \begin{aligned}
        \label{eq:AhA}
        A_{\tau\pi}(z,D_z) &:= \tau(z)^{-1} ( \pi(z) D_z - \pi'(z) ), \\
        \Bal_{\pi\tau}(z,D_z) &:= (1-z^2) A_{\pi\tau}(z,D_z) -
        (2\al+1)z\tau(z)\pi(z)^{-1}.
    \end{aligned}
\end{equation}
It will be useful to express the above
operators in terms of Wronskian determinants by setting
\begin{equation}
  \label{eq:hpihtau}
  \begin{aligned}
    \hpi(z) &:=  (1-z^2)^{-\alpha-\frac32}\pi(z),\\
    \htau(z) &:= (1-z^2)^{-\alpha-\frac12} \tau(z),
  \end{aligned}
\end{equation}
so that
\begin{align*}
  &A_{\tau\pi} y = \tau^{-1} \Wr[\pi,y], \\
  &\Bal_{\pi\tau} y = \hpi^{-1} \Wr[\htau,y].
\end{align*}
Observe that it would be sufficient to define only $A_{\tau\pi}$ since
the two operators are related by
\begin{gather}\label{eq:idbetweenAB}
    \Bal_{\pi\tau} = A_{\hpi\htau} .
\end{gather}
However, defining both operators separately will make for simpler notation going forward.
With these first order operators, we can now describe the factorization of an $X$-Gegenabuer operator \eqref{eq:GegOp} in the following proposition.
\begin{prop}
    \label{prop:formsd}
    Let $\alpha\in\R$, $\tau(z)$ be a polynomial,
    and $\Tal_\tau$ be an exceptional Gegenabuer operator as in \eqref{eq:GegOp}. Assume that  $\pi(z)$ is an eigenpolynomial
     of $\Tal_\tau$ with eigenvalue
    $\lambda$, i.e. $\Tal_\tau\pi=\lambda\pi$ . We then have the following factorizations:
    \begin{equation}
        \label{eq:sdfac}
        \begin{aligned}
          \Bal_{\pi\tau} A_{\tau\pi}
          &=\Tal_\tau- \lambda , \\
          A_{\tau\pi} \Bal_{\pi\tau} &=
          T^{(\alpha+1)}_\pi - \hlambda,
        \end{aligned}
      \end{equation}
      where
      \begin{equation}
        \label{eq:hlambda}
         \hlambda = \lambda+ 2\alpha + 1.
      \end{equation}
\end{prop}
\noindent
\begin{proof}
  The proof follows by a straightforward computation.
\end{proof}

We say that the transformation
$\Tal_\tau \to T^{(\al+1)}_{\pi}-(2\alpha+1)$ is a formally
\emph{state-deleting} Darboux transformation, because the second
operator no longer has an eigenpolynomial at the eigenvalue $\lambda$.
Likewise, we refer to the transformation
$T^{(\alpha+1)}_\pi \to \Tal_\tau{+(2\alpha+1)}$ as a formally \emph{state-adding}
Darboux transformation, because the operator $\Tal_\tau$ gains an extra
polynomial eigenfunction $\pi$ at eigenvalue $\lambda$.  The following
proposition describes the factorization eigenfunction for this dual
factorization.
\begin{prop}
  \label{prop:formsa}
  Let $\alpha\in\mathbb{R}$, $\tau(z)$ and $\pi(z)$ be non-zero polynomials and let $\htau$ be as
  in \eqref{eq:hpihtau}.
  Suppose that the following eigenvalue equation holds:
  $T^{(\alpha+1)}_\pi \htau = \hlambda\htau$.  Then  the factorizations \eqref{eq:sdfac}
   also hold.
\end{prop}

\begin{proof}
   By direct calculation, we obtain that
 \begin{equation}
   \label{eq:piThtau}
   \pi T_\pi^{(\alpha+1)} \htau=\htau T_\tau^{(\alpha)} \pi +(2\alpha+1)\pi\htau.
  \end{equation}
Hence, $\Tal_\tau \pi = \lambda \pi$ is equivalent to
  $T^{(\alpha+1)}_\pi \htau = \hlambda \htau$. 
  This, together with \eqref{eq:idbetweenAB}, implies the result. 
%

\end{proof}

\subsection{Confluent Darboux transformations of exceptional Gegenbauer operators} \hfill
\label{subsec:CDTGeg}

In this section, we apply the theory of confluent Darboux transformations developed in Section~\ref{sec:CDT} to exceptional Gegenbauer operators. With the factorizations introduced in the previous section, we will realize a CDT of an exceptional
Gegenbauer operator $T_0$ by a state-deleting
transformation $T_0 \to T_1$ followed by a 1-parameter family of
state-adding transformations $T_1\to T_2$ at the same eigenvalue. Next we derive certain recursive formulas that connect the $\tau$-functions and eigenpolynomials of two $X$-Gegenbauer operators connected by a CDT.

Starting from the classical Gegenbauer operator we will construct a chain of CDTs and a recursive construction of exceptional Gegenbauer operators of second kind.
This recursive construction will be discussed in more detail in the
following section.


\begin{remark}
While the definitions of X-Gegenbauer operators and their factorizations in the previous Section hold for any $\alpha\in\R$, we need to assume that certain integrals like \eqref{eq:rhoij} define rational functions, which requires that $\alpha\in \Nz+\frac12$ from here on.
\end{remark}

Suppose moreover  that $\tau(z)$ is a polynomial and $\Tal_\tau$ is an exceptional
Gegenbauer operator, as per \eqref{eq:GegOp}, with eigenpolynomials
$\{ \pi_i \}_{i \in \N_0}$. We define the functions
\begin{align}
    \label{eq:rhoij}
  &\rho_{ij}(z) := \int_{-1}^z  \pi_{i}(u) \pi_{j}(u)
    \Wal_\tau(u) du , \quad i,j\in \Nz;\\  
    \label{eq:taum}
  &\tau_m(z,t) := \tau(z) ( 1 + t \rho_{mm}(z) ) , \quad m\in \Nz;\\
    \label{eq:pimi}
    &\pi_{m;i}(z,t) := ( 1 + t \rho_{mm}(z) ) \pi_i(z) - t
      \rho_{im}(z) \pi_m(z) ,   \quad i,m\in \Nz,
\end{align}
where the integral in \eqref{eq:rhoij} denotes a formal
anti-derivative that vanishes at $z=-1$.
\begin{remark}
Throughout this Section, we  assume that $\tau_m(z)$ and $\pi_{m;i}(z,t)$ are
polynomials in $z$ and $\rho_{i j}(z)$ is a
rational function of $z$. In principle, this assumption seems a strong requirement when looking at \eqref{eq:rhoij} and \eqref{eq:Wtaldef}. However, we will see in Section~\ref{sec:xgegen} that whenever these quantities are connected recursively to the classical Gegenbauer operator and polynomials, there exist matrix formulas that establish the polynomial character of $\tau_m(z)$ and $\pi_{m;i}(z,t)$ and the rational character of  $\rho_{i j}(z)$ by construction.
\end{remark}


\begin{prop}
    \label{prop:CDT}
    For a given $m \in \Nz$, the operators $\Tal_\tau, \Tal_{\tau_m}$
    are related by a confluent Darboux transformation generated by the
    seed functions $\{\pi_m,\tor{\pi_m^{(1)}\}}$.
\end{prop}

\begin{proof}
  Set $T_0 = \Tal_\tau$, and recall that, by assumption,
  $T_0 \pi_m = \lambda_m \pi_m$. We set
    \begin{gather*}
        T_1 := T^{(\al+1)}_{\pi_m} - (2\al + 1),\quad T_2 := \Tal_{\tau_m},\\
        A_1 := A_{\tau,\pi_m},\quad A_2 := \Bal_{\pi_m, \tau_m}
    \end{gather*}
    By Proposition ~\ref{prop:formsd}, $T_0\to T_1$ is a state-deleting
    Darboux transformation. Consequently, 
    \begin{equation}
        \label{eq:A1T01}
        A_1 T_0 = T_1 A_1.
    \end{equation}

    We now claim that the transformation $T_1\to T_2$ is a 1-parameter
    family of state-adding Darboux transformations. By inspection of
    \eqref{eq:GegOp}, we find that the relevant coefficient functions
    of $T_0$ are
    \begin{gather*}
        p := 1-z^2,\quad q_0 := -2 (1-z^2) \frac{\tau_z}{\tau} - (2\al+1)z .
    \end{gather*}
   Then the function $\mu(z)$ in \eqref{eq:mudef} satisfies
    \begin{gather*}
        \frac{\mu'}{\mu} = -\frac{q_0}{p} = 2\frac{\tau_z}{\tau} +
        \frac{(2\al+1)z}{1-z^2} , 
    \end{gather*}
    and hence
    \begin{gather*}
        \mu(z) :=(1-z^2)^{-\al-\frac12} \tau(z)^2 .
    \end{gather*}

    Following \eqref{eq:phi2int}, we construct a generalized
    eigenfunction of $T_0$ as shown in Section ~\ref{sec:CDT}. We
    define
    \begin{align}
        \label{eq:pim1def}
        &\pi_m^{(1)}(z;t) := \pi_m(z) \int_{-1}^z (1+t\rho_{mm}(s)) \htau(s)^2 W_{\pi_m}^{(\al+1)}(s) ds,
        \intertext{where} \nonumber
        &\htau(z) := (1-z^2)^{-\al-1/2}\tau(z)
    \end{align}
    and $W_{\pi_m}^{(\alpha+1)}$ is given in \eqref{eq:Wtaldef}.
    Observe that $\Wal_\tau = (p\mu)^{-1}$, in agreement with
    \eqref{eq:Wpmu}. 
    The definitions of
    $\rho_{mm},\pi_m^{(1)}, \htau$ agree with the definitions of
    $\rho,\phi^{(1)},\psi_1^\perp$ in \eqref{eq:rhodef}, \eqref{eq:phi2int} and
    \eqref{eq:psi1perp}, respectively.  Consequently, as shown in Section ~\ref{sec:CDT}, $\pi_m^{(1)}$ is a
    first-order generalized eigenfunction of $T_0$, because it
    satisfies the equation
    \begin{gather}
        (T_0 - \lambda_m) \pi_m^{(1)}  = t \pi_m .
    \end{gather}

    
    A direct calculation shows that
    \begin{gather*}
        A_1 \pi_m^{(1)} = \htau_m.
    \end{gather*}
    where
    \begin{equation}
        \label{eq:tautdef}
      \htau_m(z,t) := (1-z^2)^{-\al-1/2} \tau_m(z,t) =\htau(z)(1 + t
       \rho_{mm}(z)). 
    \end{equation}
    Hence, by \eqref{eq:A1T01},
    \begin{gather*}
      T_1 \htau_m = T_1 A_1 \pi^{(1)}_m = A_1 T_0 \pi^{(1)}_m =
      \lambda_m \htau_m;
    \end{gather*}
    i.e., $\htau_m$ is a factorization eigenfunction of $T_1$ at
    $\lambda_m$. By Proposition ~\ref{prop:formsa}, we have thus
    \[ \begin{aligned}
        \Bal_{\tau_m, \pi} A_{\pi, \tau_m}  &= \Tal_{\tau_m} - \lambda\\
        A_{\pi, \tau_m} \Bal_{\tau_m, \pi}&= T^{(\alpha+1)}_{\pi}
        -2\alpha-1 -\lambda.
      \end{aligned}
    \]
    We conclude therefore that $T_2 A_2 = A_2 T_1$, as was to be shown.
  \end{proof}

  \begin{prop}
    \label{prop:eigenpimi}
    The polynomials $\pi_{m,i}(z;t)$  as defined in \eqref{eq:pimi}
    are eigenfunctions of $\Tal_{\tau_m}$.
  \end{prop}
\begin{proof}  

  Combining the intertwining relations in the preceding proof, we have
  the second-order intertwining relation
    \begin{equation}
        \label{eq:A21T0T2}
        A_{21} T_0 = T_2 A_{21},\quad \text{where } A_{21}:= A_2 A_1.
    \end{equation}
    Our first claim is that
    \begin{equation}
      \label{eq:A21pii}
        A_{21} \pi_i = (\lambda_i-\lambda_m) \pi_{m,i},\quad i\neq m.
    \end{equation}

    Let $i \in \Nz$ be given. To establish the previous equation, we consider $A_1$ and $A_2$ in terms of their Wronskian formulations:
    \begin{align*}
        A_1y & = A_{\tau\pi_m}y = \tau^{-1}\Wr[\pi_m,y] ,\\
        A_2y &= \hA^{(\al)}_{\pi_m\tau_m}y = \hpi_m^{-1}\Wr[\htau_m,y] ,
        \intertext{where}
        \hpi_m(z) &= (1-z^2)^{-\al-\frac32} \pi_m(z).
    \end{align*}
    Lagrange's identity \eqref{eq:lagid}  implies that
    \begin{gather*}
        \Wr[\pi_m,\pi_i]=(\lambda_i-\lambda_m)\mu\, \rho_{mi}.
    \end{gather*}
    Hence,
    \begin{gather*}
        A_1 \pi_i = \tau^{-1} \Wr[ \pi_m, \pi_i ] = (\lambda_i - \lambda_m) \htau \rho_{mi} .
    \end{gather*}
    We define $B_1 := \Bal_{\pi_m, \tau}$ and note that by
    Proposition \ref{prop:formsd} we have the factorization
    $T_0 = B_1 A_1 + \lambda_m$. By the linearity of the Wronskian, we
    also have
    \begin{align*}
        A_2 y &= \hpi_m^{-1} \Wr[ \htau_m, y ] = \hpi_m^{-1} \Wr[ \htau (1+t\rho_{mm}), y ] \\
        &= \hpi_m^{-1} \Wr[ \htau, y ] + \hpi_m^{-1} \Wr[ t \htau \rho_{mm}, y ] \\
        &= B_1 y + t \hpi_m^{-1} \Wr[ \htau \rho_{mm}, y ]
    \end{align*}
    Finally, we have that
    \begin{align*}
      A_{21} \pi_i
      &= B_1 A_1 \pi_i +t (\lambda_i-\lambda_m)
        \hpi_m^{-1} \Wr[\htau \rho_{mm}, \htau
        \rho_{mi}]\\ 
      &=(T_0-\lambda_m)\pi_i + (\lambda_i-\lambda_m)t \htau^2
        \hpi_m^{-1} \Wr[\rho_{mm},\rho_{mi}] \\ 
      &= (\lambda_i-\lambda_m)(\pi_i + t\htau^2\hpi_m^{-1} (\rho_{mm}
        \pi_m \pi_i {W^{(\alpha)}_\tau} - \rho_{mi} \pi_m^2
        {W^{(\alpha)}_\tau})\\ 
      &= (\lambda_i-\lambda_m)(\pi_i + t(\rho_{mm} \pi_i - \rho_{mi}
        \pi_{{m}})). 
    \end{align*}
    This establishes \eqref{eq:A21pii}.

    Next suppose that $i\ne m$.  Then, by \eqref{eq:A21T0T2}, it
    follows that
    \begin{gather}
        \label{eq:T2pimi}
        T_2 \pi_{m;i} = (\lambda_i-\lambda_m)^{-1} T_2 A_{21} \pi_i
         = (\lambda_i-\lambda_m)^{-1}  A_{21} T_0 \pi_i
        \lambda_i \pi_{m;i}. 
      \end{gather}
      Finally, observe that $\pi_{m,m} = \pi_m$ and recall that
      \[ T_2 = \Bal_{\pi_m,\tau_m} A_{\tau_m, \pi_m} + \lambda_m.\]
      Hence, $T_2 \pi_m = \lambda_m \pi_m$, as was to be shown.
\end{proof}


The last proposition of this section shows that the CDT of an
exceptional Gegenbauer family falls into the same class provided we
impose suitable bounds on the introduced parameter $t$. The form of
the norming constants of the new eigenpolynomials relative to the
weight $\Wal_{\tau_m}$ follow as a direct corollary to this
proposition. We denote the norm of $\pi_i$ by $\nu_i$, and similarly
the norm of $\pi_{m;i}$ is $\nu_{m;i}$. Explicitly, we have
\begin{align}
  \label{eq:nui}
  \nu_i
  &:= \int_{-1}^1 \pi_{i}(s)^2 \Wal_{\tau}(s) ds = \rho_{ii}(1) ,
    \quad i\in \Nz;\\
  \label{eq:numi}
  \nu_{m;i}
  &:= \int_{-1}^1 \pi_{m;i}(s)^2 \Wal_{\tau_m}(s) ds  =
    \rho_{m;ii}(1), \quad m,i \in \Nz;\intertext{where}
    \label{eq:rhomij}
    \rho_{m;ij}(z)
  &:=   \int_{-1}^z \pi_{m; i}(u; t)
    \pi_{m; j}(u; t)\Wal_{\tau_m}(u) du   ,\quad m,i,j\in \Nz.
\end{align}
\noindent
The following proposition exhibits a recursive formulation for the
functions $\rho_{m;ij}$.  As a consequence, if the $\rho_{ij}(z)$ are
 rational, then so are $\rho_{m;ij}(z)$, i.e. rationality is preserved by a CDT.
\begin{prop}
    \label{prop:rho2step}
    Let $\rho_{i j}, \tau_m,\pi_{m;i},\rho_{m;ij}$ be defined by
    \eqref{eq:rhoij}--\eqref{eq:pimi} and \eqref{eq:rhomij}. Then,
    \begin{align}
        \label{eq:rho2step}
      \rho_{m;ij}(z,t)=\rho_{i j}(z) - \frac{t \rho_{m i}(z) \rho_{m
      j}(z)}{1+t \rho_{mm}(z)},\quad m,i,j\in \Nz. 
    \end{align}
\end{prop}

\begin{proof}
    Observe that
    \begin{align*}
      &\lp \rho_{i j}- \frac{t\rho_{m i} \rho_{m j}}{1+t
        \rho_{mm}}\rp' = \rho_{i j}' - \frac{t \rho_{m i}'
        \rho_{m j}}{1+t \rho_{mm}} - \frac{t \rho_{m
        i} \rho_{m j}'}{1+t \rho_{mm}} + \frac{t^2
        \rho_{m i} \rho_{m j}}{(1+t \rho_{mm})^2}
        \rho_{mm}' \\  
      &\quad = \Big( (1+t \rho_{mm})^2 \pi_{i} \pi_{j} - t \pi_m
        (1+t \rho_{mm}) (\pi_{i} \rho_{m j} + \pi_{j} \rho_{m
        i}) + t^2 \rho_{m i} \rho_{m j} \pi_m^2 \Big)
        \frac{\Wal_{\tau}}{(1+t \rho_{mm})^2} \\ 
      &\quad = \Big( ( 1 + t \rho_{mm} ) \pi_{i} - t \rho_{m i} \pi_m \Big) \Big( ( 1 + t \rho_{mm} ) \pi_{j} - t \rho_{m j} \pi_m \Big) \Wal_{\tau_m} \\
      &\quad = \pi_{m; i} \pi_{m; j} \Wal_{\tau_m}.
    \end{align*}
     The
    desired result follows then by integration since
    $\rho_{i j}(-1)=0$ by definition, and $\rho_{i j}(z)$ is rational
    by assumption. 
\end{proof}

In order to demonstrate that the $\pi_{m;i}(z)$ are exceptional Gegenbauer polynomials in the sense of Definition ~\ref{def:GegPol}, we must first establish the positivity
of $\tau_m$ on $[-1,1]$ conditioned on the positivity of $\tau$.

\begin{prop}
  \label{prop:taumpos}
  Suppose that $\tau(z)>0$ for all $z\in I= [-1,1]$.  Then, $\tau_m$
  is positive on $I$ if and only if $1+t\nu_{m}>0$.
\end{prop}
\begin{proof}
  By \eqref{eq:taum}, $\tau_m$ is positive on $I$ if and only if the
  same is true for $1+t \rho_{mm}(z)$. We know that $\rho_{mm}(z)$ is
  an increasing function, since
    \begin{gather*}
        \rho_{mm}' = \pi_m^2 \frac{(1-z^2)^{\al-1/2}}{\tau^2} >0.
    \end{gather*}
    Additionally, we notice that $\rho_{mm}$ is differentiable, and
    hence continuous, on the interval $I$. Fix $t$ and  define
    $g(z) := 1 + t \rho_{mm}(z)$, so that
    $g'(z) = t \rho_{mm}'(z)$. If $t \geq 0$, then $g$ is positive on
    $I$, since $\rho_{mm}$ is an increasing function and
    $g(-1) = 1$. However, if $t < 0$, then $g$ is a decreasing
    function on $I$. Importantly, $g$ cannot have any local
    extrema on the interval. Hence, if $g(1) = 1 + t \nu_m > 0$, then
    $g$ must be positive on all of $I$. Hence,
    $1 + t \rho_{mm}(z)$ is positive on $z \in I$ if and only if
    $1 + t\nu_m > 0$. 
\end{proof}

\begin{prop}
    \label{prop:norms}
    Let $\pi_{m;i}(z,t)$ be defined as in \eqref{eq:pimi}, and assume
    that $1+t\nu_m>0$.  Then, the norms \eqref{eq:nui} and \eqref{eq:numi}
    are related by
    \begin{equation}
      \label{eq:numi2}
      \nu_{m;i}^{-1} = \nu_m^{-1}+\delta_{im} t.
      \end{equation}
\end{prop}

\begin{proof}
    By Proposition ~\ref{prop:rho2step}, we have that
    \begin{gather*}
      \nu_{m;i} = \nu_i - \frac{t \rho_{mi}(1)^2}{1+t\rho_{mm}(1)} =
      \nu_i,\quad i\neq m,
    \end{gather*}
    Suppose that $i\neq m$. Then, $\pi_i$ and $\pi_m$ for $i \neq m$ are
    orthogonal relative to $\Wal_\tau$. Hence, $\rho_{mi}(1)=0$ and
    $\nu_{m;m} = \nu_m$.
    By definition, $\nu_{m} = \rho_{mm}(1)$.  Hence, if $i=m$, we have
    \[      
      \nu_{m;m}
      =  \nu_m - \frac{t
        \nu_m^2}{1+ t \nu_m}  = \frac{\nu_m}{1+t\nu_m} 
      = (t+\nu_m^{-1})^{-1}.
      \]
\end{proof}

\noindent
\begin{remark}
  As a consequence of \eqref{eq:numi}, we can recover the deformation
  parameter as the difference of the following norm reciprocals:
  \[
    t= \nu_{m;m}^{-1} -\nu_m^{-1}.
  \]
  Consequently,
    \begin{gather*}
        \nu_{m; m}^{-1} =  \nu_m^{-1}(1 + t \nu_m).
    \end{gather*}
    Hence, by Proposition ~\ref{prop:taumpos}, the positivity of
    $\tau_m$ is equivalent to the condition that $\nu_{m;m}>0$.
\end{remark}

\begin{prop}
  \label{prop:recortho}
 Let $\alpha\in\Nz+1/2$ and suppose that $\{\pi_i(z)\}_{i\in\Nz}$ are exceptional Gegenbauer
  polynomials with respect to the weight $\Wal_\tau(z)$.  Let
  $m\in \Nz$ and suppose that $1+ t\nu_m>0$.  Then, the family
  $\{ \pi_{m;i}(z) \}_{i\in \Nz}$ are also exceptional Gegenbauer
  polynomials with respect to the weight $\Wal_{\tau_m}(z)$.
\end{prop}

\begin{proof}
  Proposition ~\ref{prop:taumpos} establishes condition (a) of
  Definition ~\ref{def:GegPol}.  Proposition ~\ref{prop:CDT} establishes
  (b).
  It remains to prove that the completeness condition (c) also holds.  Set
  \begin{equation}
    \label{eq:tpidef}
    \begin{aligned}
    \tpi_i(z) &:= \lp\Wal_\tau(z)\rp^{\frac12} \pi_i(z) =
    (1-z^2)^{\frac{\al}{2} - \frac14} \frac{\pi(z)}{\tau(z)},\\
    \tpi_{m;i}(z) &:= \lp\Wal_{\tau_m}(z)\rp^{\frac12} \pi_{m;i}(z) =
    (1-z^2)^{\frac{\al}{2} - \frac14} \frac{\pi_{m;i}(z)}{\tau_m(z)}.
    \end{aligned}
  \end{equation}
  By assumption, the eigenpolynomials $\{\pi_i(z)\}_{i \in \Nz}$ 
  form a complete basis of $\rL^2(I,\Wal_\tau(z)dz)$. Equivalently,
  $\{ \tpi_i(z)\}_{i\in \Nz}$ are complete in $\rL^2(I,dz)$.  We seek
  to show that $\{ \tpi_{m;i}(z)\}_{i\in \Nz}$ are complete in
  $\rL^2(I,dz)$ also.

  Following an argument adapted from the appendix of
  \cite{AM80}, we observe that the completeness of $\{\tilde\pi_i\}_{i\in\Nz}$ in $\rL^2(I, dz)$ is equivalent to 
  \begin{equation}
    \label{eq:sumtpidelta}
      \sum_{i\in \Nz}\nu_i^{-1} \tpi_i(z) \tpi_i(w) = \delta(z-w),
  \end{equation}
    where the equality must be understood in distributional sense on
    $I\times I$.
    Let $f$ be  piece-wise continuous function in $I$ and set
    \begin{equation}
      \label{eq:fidef}
      [f]_i = \nu_i^{-1} \int_I\tpi_i(u) f(u) du.
    \end{equation}
    Relation \eqref{eq:sumtpidelta} then entails
    \begin{equation}
      \label{eq:sumfitpi}
      \sum_{i\in \Nz}  [f]_i \, \tpi_i(z) =      f(z),\quad \text{a.e. } z\in I.
    \end{equation}
    Moreover, if $g$ is another piece-wise continuous in $I$
    function, then
    \begin{equation}
      \label{eq:intfg}
       \int_I f(u) g(u) du =  \sum_{i\in \Nz} \nu_i [f]_i [g]_i.
    \end{equation}
    For a given $w\in I$, and a function $f(z),\; z\in I$, define the
    truncation operator
    \[ (\cT_wf)(z) := \theta(w-z) f(z),\quad z\in I,\]
    where
    \begin{equation}
      \label{eq:heaviside}
      \theta(u):=
      \begin{cases}
        1 & \text{ if } u> 0,\\
        \frac12 & \text{ if } u=0,\\
        0 & \text{ if } u<0;
      \end{cases}
    \end{equation}
    denotes the Heaviside step function.  We may now rewrite
    \eqref{eq:rhoij} as
    \[
      \rho_{i j}(w)  =  \int_{I} 
      (\cT_w\tpi_{j})(u)\, \tpi_{i}(u)\,
      du=\nu_i [ \cT_w\tpi_j]_i,\quad i, j\in \Nz,\; w\in I.
    \]
    Applying \eqref{eq:sumfitpi} then    gives 
    \begin{equation}
      \label{eq:thetawztpiz}
      \theta(w-z) \tpi_j(z) = (\cT_w\tpi_{j})(z)
      =\sum_{i\in \Nz}  \nu_i^{-1} \rho_{ij}(w)    \tpi_i(z),\quad
      z,w\in I.
    \end{equation}
    Moreover,  \eqref{eq:intfg} implies that
    \begin{align*}
      \sum_{i\in \Nz} \nu_i^{-1} \rho_{ij}(z) \rho_{ij}(w)
        &=\int_I  (\cT_w \tpi_j)(u)  (\cT_z \tpi_j)(u) du,
          \quad z,w\in I,\;j\in \Nz.
    \end{align*}
    Observe that
    \[ \theta(z) \theta(w)
      = \theta(\min(z,w)),\quad z,w\in I.\]
    Hence,
    \[ \int_I (\cT_w \tpi_j)(u) (\cT_z \tpi_j)(u) du = \int_I
      \theta(w-u)\theta(z-u) \tpi_j(u) \tpi_j(u)\, du =
      \rho_{jj}(\min(z,w)).\] Therefore,
    \begin{equation}
      \label{eq:sumrhozw}
            \sum_{i\in \Nz} \nu_i^{-1} \rho_{ij}(z) \rho_{ij}(w)
            = \theta(w-z) \rho_{jj}(z) + \theta(z-w) \rho_{jj}(w). 
    \end{equation}
    By \eqref{eq:taum}, \eqref{eq:pimi} and \eqref{eq:tpidef}, 
    \begin{align}
      \label{eq:trhomm}
      1+ t\rho_{mm}       &= \frac{\tau_m}{\tau};\\
      \label{eq:tpimi}
      \tpi_{m;i}
      &=   (1-z^2)^{\frac{\al}{2} - \frac14} \frac{1}{\tau_m} \lp
        \frac{\tau_m}{\tau} \pi_i - t\rho_{im}
        \pi_m\rp
        =\tpi_i  - t \rho_{im} \tpi_{m;m} ;\\
      \label{eq:tpim}
      \tpi_m &= (1+t\rho_{mm}) \tpi_{m;m}.
    \end{align}
    By \eqref{eq:numi},  \eqref{eq:tpimi},
    \eqref{eq:sumtpidelta}, \eqref{eq:thetawztpiz}, \eqref{eq:sumrhozw} and
    \eqref{eq:tpim}, we have 
    \begin{align*}
      \sum_{i\in \Nz}&\nu_{m;i}^{-1} \tpi_{m; i}(z) \tpi_{m; i}(w) \\
      & = t
        \tpi_{m;m}(z)\tpi_{m;m}(w)
        + \sum_i\nu_{i}^{-1} \tpi_{m;i}(z)\tpi_{m;i}(w)\\
      & = t  \tpi_{m;m}(z)\tpi_{m;m}(w)
        + \sum_i\nu_{i}^{-1} \lp
        \tpi_i(z) -t \rho_{im}(z)\tpi_{m;m}(z)\rp \lp
        \tpi_i(w) - t\rho_{im}(w) \tpi_{m;m}(w)\rp \\ 
      & = \sum_i \nu_{i}^{-1} \tpi_i(z) \tpi_i(w)+t
        \tpi_{m;m}(z)\tpi_{m;m}(w) \lp 1  + t \sum_i \nu_i^{-1}
        \rho_{im}(z)\rho_{im}(w)\rp\\ 
      &\quad -t\, \tpi_{m;m}(z) \sum_i \nu_i^{-1} \rho_{im}(z) \tpi_i(w)
        -t\, \tpi_{m;m}(w) \sum_i \nu_i^{-1} \rho_{im}(w) \tpi_i(z)\\
      & = 
        \delta(z-w)+ t  \tpi_{m;m}(z)\tpi_{m;m}(w) \lp 1+\theta(w-z) t
        \rho_{mm}(z) + \theta(z-w)  
        t\rho_{mm}(w)\rp\\
      &\quad -
        t\,\theta(z-w)\tpi_{m;m}(z)\tpi_m(w)- 
        t\,\theta(w-z) \tpi_m(z) \tpi_{m;m}(w)\\
      & = \delta(z-w)+ t  \tpi_{m;m}(z)\tpi_{m;m}(w) \Big(  1
        +\theta(w-z) t \rho_{mm}(z) + \theta(z-w)
        t\rho_{mm}(w)\\
      &\qquad\qquad
        -        \theta(z-w) (1+t \rho_{mm}(w)) 
        -        \theta(w-z) (1+t \rho_{mm}(z)) 
        \Big)\\
      & = \delta(z-w)+ t  \tpi_{m;m}(z)\tpi_{m;m}(w) \Big(
        1-\theta(w-z) - \theta(z-w)\Big)\\
      &= \delta(z-w).
    \end{align*}
We conclude that the set  $\{\tilde \pi_{m;i}\}_{i\in\Nz}$ is complete in $\rL^2(I, dz)$, and by virtue of the previously stated equivalence, the set $\{\pi_{m;i}\}_{i\in\Nz}$ is complete in $\rL^2(I, \Wal_{\tau_m}(z)dz)$.
\end{proof}

Note that in Proposition~\ref{prop:recortho} and in fact, all throughout the current Section, we have assumed that $\pi_{m,i}$ are polynomials, which is not guaranteed by their defining expressions \eqref{eq:rhoij}-\eqref{eq:pimi} and \eqref{eq:Wtaldef}, even if $\pi$ are polynomials and $\alpha\in\Nz+1/2$ (but they are certainly not polynomials if $\alpha$ is not half-integer). In the following Section we will show that when these objects are connected to the classical Gegenbauer polynomials by a chain of confluent Darboux transformations, the assumptions on polynomiality are guaranteed at each step of the chain.


\section{Exceptional Gegenbauer polynomials of the second kind}
\label{sec:xgegen}

In this Section we provide explicit formulas for the construction of exceptional Gegenbauer polynomials of the second kind and their associated operators in terms of a 
matrix whose entries involve classical Gegenbauer
polynomials. Throughout this section, we also assume that $\al\in \Nz+\frac12$ is a
positive half-integer.  Note that, with this assumption in place, the
weight $\Wal(z)$  in \eqref{eq:Wdef} becomes a polynomial.

We next define several objects that allow us to describe the exceptional Gegenbauer polynomials and operators below.
Set
\begin{equation}
  \label{eq:Rijdef}
  \rhoal_{ij}(z):= \int_{-1}^z  \Cal_i(u)
  \Cal_j(u)\Wal\!(u)\, du,   \quad i,j\in \Nz
\end{equation}
and observe that the above functions are polynomials precisely because
$\al$ is a positive half-integer.  Given an $n$--tuple $\bm\in\Nz^n$
and the associated $\bt_\bm\in\R^n$, we then define
$\cRal_{\bm} = \cRal_{\bm}(z;\bt_\bm)$ as the $n\times n$ matrix with
polynomial entries given by
\begin{align}
  \label{eq:cRalkl}
  [\cRal_{\bm}]_{k\ell} = \delta_{k\ell}+t_{m_\ell}
  \rhoal_{m_km_\ell}(z), \quad k,\ell\in\{1,\dots, n\}, 
\end{align} 
We denote its determinant by
\begin{gather}
  \label{eq:taudef}
  \taual_{\bm} := \det \cRal_\bm.
\end{gather}
Next, define the $n$--tuple of polynomials
\begin{gather}
  \label{eq:Qmdef}
  \lp\bQal{\bm}{}\rp^T:= \taual_{\bm} \lp\cRal_{\bm}\rp^{-1} \lp
  \Cal_{m_1},\ldots, \Cal_{m_n}\rp^T,
\end{gather}
We are now ready to define the fundamental objects of this Section.
\begin{definition}
Let $\bm=(m_1,\dots,m_n)\in \Nz^n$ be a tuple of distinct integers and $\bQal{\bm}$ the $n$-tuple of polynomials defined by \eqref{eq:Rijdef}-\eqref{eq:Qmdef}. We define the \emph{exceptional Gegenbauer polynomials of the second kind} associated to $\bm$ as
\begin{align}
  \label{eq:Cmidef}
  \Cal_{\bm;i} := \big[ \bQal{(\bm, i)} \big]_{n+1},\quad i\in \Nz.
\end{align}
\end{definition}

\begin{remark}
Note that, by construction, $\taual_\bm=\taual_{\bm}(z;\bt_\bm)$ is
invariant with respect to permutations of the indices
$\bm=(m_1,\ldots, m_n)$ and that
$\bQal{\bm} = \bQal{\bm}(z;\bt_{\bm})$ is equivariant with respect to
such permutations. In addition,
$\Cal_{\bm;i}= \Cal_{\bm;i}(z;\bt_\bm)$ is symmetric in $\bm$ and does
not depend on $t_i$ since
$\taual_{(\bm,i)} \big[ \lp\cRal_{(\bm, i)}\rp^{-1} \big]_{n+1,j}$
correspond to the minors of the last column of $\cRal_{(\bm,i)}$, the
only column where $t_i$ appears.
\end{remark}

The main result of this section states that the polynomials
$\big\{ \Cal_{\bm; i}(z; \bt_\bm)\big\}_{i \in \Nz}$ defined above are indeed exceptional Gegenbauer polynomials (in the sense of Definition~\ref{def:GegPol}), provided the
real parameters $\bt_\bm$ satisfy certain constraints to ensure that
$\taual_\bm(z;t_\bm)$ is positive on $z\in[-1,1]$. We first state that the polynomial $\tau_{\bm}^{(\alpha)}$ inserted in  \eqref{eq:GegOp}  is an exceptional Gegenbauer operator in the sense of Definition~\ref{def:GegOp}.
\begin{thm}
    \label{thm:Teigen}
    Let $\alpha\in\Nz+\frac 12$ and $\bm\in \Nz^n$. Consider the $n$-parameter family of
    operators $\Tal_{\bm} := \Tal_{\taual_{\bm}}$ given by
    \eqref{eq:GegOp} and \eqref{eq:taudef}. For each value of the
    parameters $\bt_{\bm}$, this operator is an exceptional Gegenbauer
    operator that satisfies
    \begin{equation}
        \label{eq:Teigen}
        \Tal_{\bm}  \Cal_{\bm;i} = \lambda_i  \Cal_{\bm;i},\quad i\in \Nz,
    \end{equation}
    with $ \lambda_i=-i(2\al+i)$.
\end{thm}

\noindent
The following theorem provides necessary and sufficient conditions for the polynomials $\{ \Cal_{\bm;i}\}_{i\in\Nz}$ to be a family of exceptional Gegenbauer polynomials according to Definition~\ref{def:GegPol}. If these conditions hold the family
$\{ \Cal_{\bm; i}(z; \bt_\bm)\}_{i \in \Nz}$ is orthogonal and
complete, like their classical counterparts.

\begin{thm}
    \label{thm:Portho}
    Let $\alpha\in\Nz+\frac 12$ and $\bm\in \Nz^n$ with $m_1,\ldots, m_n$ distinct. Then the polynomial
    $\taual_\bm(z;\bt_\bm)$ in \eqref{eq:taudef} has no zeros on
    $z\in [-1,1]$ if and only if
    \begin{equation}
        \label{eq:tconstraints}
        t_{m_j} > -\lp \nual_{m_j}\rp^{-1},\quad j=1,\ldots, n,
    \end{equation}
    with $ \nual_{m_j}$ as in  \eqref{eq:nuidef}.
    If the above conditions hold, then
    $\big\{ \Cal_{\bm; i}(z; \bt_\bm)\big\}_{i \in \Nz}$ are exceptional
    Gegenbauer polynomials with weight $\Wal_{\tau_m}$ and  norms 
    given by
    \begin{equation}
        \label{eq:norm}
        \int_{I} \lp \Cal_{\bm; i}(u)\rp^2 \Wal_{\tau_m}(u)du \
        = \frac{\nual_i}{1+ \delta_{i,\bm}  t_i\nual_i}
      \end{equation}
      where
      \begin{equation}
        \label{eq:deltaibm}
        \delta_{i,\bm} :=
        \begin{cases}
          1 & \text{ if } i \in \{ m_1,\ldots, m_n \};\\
          0 & \text{ otherwise.}
        \end{cases}
      \end{equation}
\end{thm}

As mentioned above, the degree of the $i$-th exceptional Gegenbauer
polynomial $ \Cal_{\bm;i}$ is not necessarily $i$. The next
proposition provides this result.  It is also worth noting that, as
opposed to the exceptional families of the first kind, the degree
sequence of the exceptional Gegenbauer polynomials is not an
increasing sequence, which is further evidence of the different
construction.
\begin{prop}
    \label{prop:degPtau}
    Let $\alpha\in\Nz+\frac 12$ and $\bm\in \Nz^n$ with $m_1,\ldots, m_n$ distinct. Let
    $\taual_\bm, \Cal_{\bm;i}$ be as defined in \eqref{eq:taudef} and
    \eqref{eq:Cmidef}. Then,
    \begin{align} 
        \label{eq:degtau}
        &\deg_z \taual_{\bm} = 2(m_1+\ldots + m_n+\al n),\\
        \label{eq:degPn}
        &\deg_z \Cal_{\bm;i} = 2(m_1+\ldots +
          m_n+\al n)+i-  2\delta_{i,\bm}(i+\al),\quad i\in \Nz. 
    \end{align}
    Moreover,
    \begin{equation}
        \label{eq:Pmi}
        \Cal_{\bm;m_k} = \Cal_{m_1,\ldots, \widehat{m_k},
          \ldots, m_n;m_k} ,\quad k=1,\ldots, n,
    \end{equation}
    where the hat symbol denotes the omission of the $k\supth$ entry
    of $\bm$.
\end{prop}
\begin{remark}
  From Proposition ~\ref{prop:degPtau}, we see that the codimension
  (number of missing degrees) of the exceptional Gegenbauer family
  indexed by $\bm=(m_1,\dots,m_n)$ is $2(m_1+\cdots+m_n+\alpha n)$. As
  it happens for all
  exceptional polynomials \cite{GFGM19}, this coincides with the degree of $\taual_\bm$.
\end{remark}

\begin{remark}
  Below we will show that the $n$-parameter family
  $\big\{\Cal_{\bm;i}(z;t_\bm)\big\}_{i\in\Nz}$ is the result of applying $n$ confluent Darboux transformations to the classical family
  $\big\{\Cal_i(z)\big\}_{i\in\Nz}$.  Identity \eqref{eq:norm} then tells us that a single
  CDT  leaves invariant all but one of the norming
  constants.   By contrast, identity \eqref{eq:Pmi} tells us that after a
  single CDT there is exactly one polynomial that remains the same but
  whose norm undergoes a change.  It is this phenomenon that accounts
  for the Kronecker delta term in \eqref{eq:degPn}.
\end{remark}

\begin{remark}
  The degree formulas \eqref{eq:degtau}, \eqref{eq:degPn} allow us to
  relate the eigenvalue formula \eqref{eq:lambdaform} with the value
  $\lambda_i = -i(2\alpha+i)$ given in Theorem ~\ref{thm:Teigen}.  Indeed if
  $i\notin \{ m_1,\ldots, m_n\}$ then
  $\deg \Cal_{\bm,i} - \deg\taual_\bm = i$ in full agreement with
  \eqref{eq:lambdaform}.  By contrast, if $i=m_k$ for some $k\in \{
  1,\ldots, n\}$ then
  \[ d_i := \deg \Cal_{\bm,i} - \deg\taual_\bm =  -i - 2\al\]
 and therefore
  \[ d_i(2\al+d_i) = (i + 2\al) i.\]
  Hence, \eqref{eq:lambdaform} is correct in this case also.
\end{remark}
 
In Theorem ~\ref{thm:Portho} and Proposition ~\ref{prop:deglamdai} we have considered the case when $\bm=(m_1,\dots,m_n)$ contains
distinct indices. Let us now show that this choice entails no loss of
generality.  Indeed, the repeated application of a  confluent
Darboux transformation at the same eigenvalue only serves to modify
the deformation parameter.

\begin{prop}
    \label{prop:duplicates}
    Let $\alpha\in\Nz+\frac 12$, $\bm\in\Nz^n$ and let $\taual_\bm,\Cal_{\bm;i}$ be as defined
    in \eqref{eq:taudef} and \eqref{eq:Cmidef}. Then, for any
    $j \in \Nz$, we have
    \begin{align*}
      &\taual_{(\bm,j,j)}(z;\tor( \bt_\bm,t_j,t_j'\tor)) =
        \taual_{(\bm,j)}(z;\tor( \bt_\bm,t_j+t'_j\tor)),\\
      &\Cal_{(\bm,j,j);i}(z;( \bt_\bm, t_j,t_j')) =
        \Cal_{(\bm,j);i}(z;(\tor( \bt_\bm, t_j+t'_j)). 
    \end{align*}
\end{prop}
\noindent

\subsection{Proof of theorems}
\label{sec:proofs}

In this section we provide proofs for all of the Theorems and Propositions in this Section. 
\begin{enumerate}
\item[1.]  We define polynomials $\ttau_j$, $\tpi_{0;i}$  and rational
functions $\trho_{j;i_1i_2}$ recursively, starting the recursion at the objects
corresponding to the classical Gegenbauer Sturm--Liouville problem.
\item[2.] We show that these recursion formulas describe a multi-step confluent Darboux transformation.
\item[3.] We show in Proposition~\ref{prop:taurec} that
the recursively defined objects coincide with those obtained via the matrix-based
definitions \eqref{eq:Rijdef}-\eqref{eq:Cmidef}. 
\item[4.] Since the objects defined by the matrix formulas are polynomial by construction, we can dispense with the rationality and polynomiality assumptions made at
the beginning of Section ~\ref{subsec:CDTGeg}.
\item[5.]  Propositions \ref{prop:CDT}--\ref{prop:recortho} then ensure that at each step of
the recursion we have an exceptional Gegenbauer Sturm--Liouville
problem, provided the parameters are chosen in the right range.
\end{enumerate}

Fix $\bm=\bm_n=(m_1, \dots, m_n) \in \Nz^n$ and a positive
half-integer $\al\in \Nz+\frac12$.  For $j=0,1,\ldots, n$ let
$\bm_{j}\in \Nz^n$ denote the initial segment of $\bm$; i.e.,
$\bm_j=(m_1, \dots, m_j)$. Note that, throughout this section, we are
going to omit the explicit dependence on $z$ and $\bt_\bm$ of the
objects, which must be understood from the dependence on $\bm$, i.e.,
we will write $\cRal_\bm$ instead of $\cRal_\bm(z;\bt_\bm)$.  In order
to simplify the notation, we will sometimes make the dependence of the
various objects on $\alpha$ implicit rather than explicit.

We start the recursion at $j=0$ by setting
\begin{align*}
  \trho_{0;i_1 i_2} := \rhoal_{i_1 i_2} ,\quad
  \ttau_{0} := 1,\quad
  \tpi_{0;i} := \Cal_i,
\end{align*}
where $\rhoal_{i_1 i_2}(z)$ is given by \eqref{eq:Rijdef} and
$\Cal_i(z)$ are the classical Gegenbauer polynomials \eqref{eq:gegpoly}.  For
$j=1, \dots, n$ we then define
\begin{align}
  \label{eq:taurec}
  \ttau_j
  &= \lp 1+t_{m_j} \trho_{j-1;m_j m_j}\rp \ttau_{j-1},\\
  \label{eq:Prec}
  \tpi_{j;i}
  &= \lp 1+t_{m_j} \trho_{j-1;m_jm_j} \rp \tpi_{j-1;i}
    - t_{m_j}\trho_{j-1;im_j}\tpi_{j-1;m_j},\quad i\in \Nz;\\
    \label{eq:Rndef}
  \trho_{j;i_1i_2}
  &= \trho_{j-1;i_1i_2} - \frac{t_{m_{j}} \trho_{j-1;i_1m_{j}}
    \trho_{j-1;i_2m_j}}{1+ t_{m_j} \trho_{j-1;m_j m_j}},\quad i_1,i_2\in \Nz.
\end{align}
These recursive definitions match the formulas \eqref{eq:taum},
\eqref{eq:pimi} and \eqref{eq:rho2step}.  Thus, in effect we are
defining the objects associated with an $n$-step confluent Darboux
transformation applied to classical Gegenbauer operators.

\begin{prop}
Let $i_1, i_2, j\in\Nz$ and $\trho_{j;i_1i_2 }$, $\tpi_{j,i}$ be as in \eqref{eq:Rndef}, \eqref{eq:Prec}.
  We have that
  \begin{align}
    \label{eq:Rint}
    \trho_{j;i_1i_2}(z)
    & := \int_{-1}^z \tpi_{j,i_1}(u)
      \tpi_{j,i_2}(u)\tW_j(u) du,\\ \intertext{where} \nonumber
    \tW_j(z)
    & := (1-z^2)^{\al-\frac12} \ttau_j(z)^{-2},
  \end{align}
  where the integral denotes an anti-derivative that vanishes at
  $z=-1$
\end{prop}
\begin{proof}
  The proof follows directly from  \eqref{eq:rhomij} and  Proposition \ref{prop:rho2step}.
\end{proof}

\begin{prop}\label{prop:taurec}
Let $\taual_{\bm}$ and $\Cal_{\bm; i} $ be as in \eqref{eq:taudef}, \eqref{eq:Cmidef} and let $\tilde \tau_n$, $\tpi_{n;i}$ be as in \eqref{eq:taurec}, \eqref{eq:Prec}
  Then, 
  \begin{align}
    \label{eq:ttautau}
    \taual_{\bm} &= \ttau_n,\\
    \label{eq:tpp}
    \Cal_{\bm; i} &= \tpi_{n;i},\quad i\in \Nz.
  \end{align}
 
\end{prop}
\begin{proof}
  The proof follows the same argument as the analogous result for exceptional Legendre polynomials, see \cite[Proposition 5]{GFGM21}.
\end{proof}

As a direct consequence of \eqref{eq:ttautau} and \eqref{eq:tpp} we see
that the recursively defined $\ttau_j, \tpi_{j;i}$ are polynomials for
each $j=1,\ldots, n$.  We have also established that the
antiderivative in the RHS of \eqref{eq:Rint} describes a rational
function.  This allows us to dispense with the polynomiality and
rationality assumptions used in Section ~\ref{subsec:CDTGeg}.

\begin{proof}[Proof of Theorem ~\ref{thm:Teigen}]
  Starting form the classical Gegenbauer operator, the application of
  a rational confluent Darboux transformation indexed by an integer
  $m_i$ introduces an extra real parameter $t_{m_i}$. Proposition
  ~\ref{prop:taurec} establishes the equivalence of the objects defined
  by the CDT recursion \eqref{eq:taurec}, \eqref{eq:Prec} and the
  matrix-based definitions \eqref{eq:taudef}, \eqref{eq:Cmidef}.  This
  allows us to apply Proposition ~\ref{prop:eigenpimi}. The eigenvalue
  relation \eqref{eq:Teigen} follows immediately.
\end{proof}

  \begin{proof}[Proof of Theorem~\ref{thm:Portho}]
  First of all, it is clear by construction that the objects $C^{(\alpha)}_{\bm;i}$ are polynomials.
  In virtue of Proposition ~\ref{prop:taurec}, the results of Section ~\ref{sec:Gegen} can be exploited.
    We recursively define for $i\in\Nz$     \[ \tnu_{0,i} = \nual_i,\quad \lp\tnu_{j,i}\rp^{-1} =
      \lp\tnu_{j-1,i}\rp^{-1}+ \delta_{im_j}t_{m_j} ,\quad j=1,\ldots,
      n.\] By Propositions ~\ref{prop:taumpos} and ~\ref{prop:norms}, condition (a) is satisfied, i.e.
    $\ttau_j(z)>0$ for $z\in I$, if and only if
    $\tnu_{j,m_j}>0$.  By the above definition,
    \[ \tnu_{j,m_j} = \frac{\nual_{m_j}}{1+ t_{m_j}\nual_{m_j}}.\]
    Therefore, $\tilde\nu_{j,m_j}>0$ for $ j=1,\ldots, n$ if and only if
    \eqref{eq:tconstraints} holds.  Relation \eqref{eq:norm} also
    follows by Proposition ~\ref{prop:norms} and a similar induction
    argument.  Finally, the completeness condition (c) follows by
    induction with  Proposition ~\ref{prop:recortho} serving as the
    inductive step.    
  \end{proof}

We conclude this section by proving the remaining results on the degrees of the polynomials and the case of repeated indices.
  \begin{proof}[Proof of Proposition ~\ref{prop:degPtau}]
    By \eqref{eq:Prec}, we have
    \[ \tpi_{j;m_j} = \tpi_{j-1,m_j},\quad j=1,\ldots, n.\]
    Identity \eqref{eq:Pmi} then follows by \eqref{eq:tpp}.
    Thanks to \eqref{eq:Pmi} no generality is lost by assuming that
    $i\notin \{ m_1,\ldots, m_n\}$.  We use induction to show that
    \begin{align}
      \label{eq:degttau}
      \deg \ttau_j
      &= 2(m_1+\cdots + m_j+ \al j),\quad  j=0,1,\ldots, n.\\
      \label{eq:degtpi}
      \deg \tpi_{j;i}
      &=
      \deg \ttau_j + i.
    \end{align}
    The desired relations \eqref{eq:degtau}, \eqref{eq:degPn} then
    follow by \eqref{eq:ttautau} and \eqref{eq:tpp}.

    By inspection, \eqref{eq:degttau} and \eqref{eq:degtpi} hold for
    $j=0$.  Assume that these relations hold for a given $j<n$.  By
    \eqref{eq:Rint},
    \[  \trho_{j,m_{j+1} m_{j+1}}(z) =  \int^z_{\tor{-1}}
      \lp\frac{\tpi_{m_{j+1}}(u)}{\ttau_j(u)}\rp^2 \tW_0(u)du .\]
    Since $m_{j+1} \notin \{ m_1, \ldots, m_j\}$, by the inductive hypothesis,
    \[ \deg_z \trho_{j,m_{j+1} m_{j+1}} = 2 m_{j+1} + 2\al,\]
    where the degree of a rational function is
    understood as the
    difference between the degrees of the numerator and
    the denominator.  Hence,
    \[ \deg_z \ttau_{j+1} = 2m_{j+1} + 2\al + \deg_z \ttau_j,\]
    which agrees with \eqref{eq:degttau} for the $j+1$ case.

    By \eqref{eq:Prec}
    and \eqref{eq:Rint},
    we have
    \[ \tpi_{j+1;i} = \tpi_{j;i} + t_{m_{j+1}} \Pi, \]
    where
    \[ \Pi(z) =\tpi_{j;i}(z) \int^z_{-1} \frac{\tpi_{j;m_{j+1}}(u)}{\ttau_j(u)}
      \frac{\tpi_{j;m_{j+1}}(u)}{\ttau_j(u)}\tW_0(u)du
      -\tpi_{j;m_{j+1}}(z) \int^z_{-1} \frac{\tpi_{j;i}(u)}{\ttau_j(u)}
      \frac{\tpi_{j;m_{j+1}}(u)}{\ttau_j(u)} \tW_0(u)du
      .\]
    Since $i\neq m_{j+1}$, the leading degree terms in the above difference do
    not cancel and hence
    \[ \deg_z \tpi_{j+1;i} = \deg_z \Pi = 2m_{j+1}+2\al+\deg
      \tpi_{j;i}.\]
    Hence \eqref{eq:degtpi} also holds for $j+1$.
  \end{proof}

\begin{proof}[Proof of Proposition ~\ref{prop:duplicates}]
  We apply Proposition ~\ref{prop:taurec} and the definitions
  \eqref{eq:taurec}, \eqref{eq:Prec} twice to obtain
    \begin{align*}
      \taual_{(\bm,j,j)}
      &= \lp 1+t'_j\rhoal_{(\bm,j);jj} \rp \taual_{(\bm,j)} \\
      & = \lp 1+t'_j\rhoal_{\bm;jj}-\frac{(t'_j)^2 \lp \rhoal_{\bm;jj}
        \rp^2}{1+t'_j \rhoal_{\bm; jj}} \rp \lp 1+t_j \rhoal_{\bm; jj} \rp
        \taual_\bm\\ 
      &= \lp 1+(t_j+t_j')\rhoal_{\bm;jj} \rp \taual_\bm\\
      &= \taual_{(\bm,j)}(z;\tor( \bt_\bm,t_j+t_j')),\\
      \Cal_{(\bm,j,j);i}
      &= \lp 1+t'_j\rhoal_{(\bm,j);jj} \rp
        \Cal_{(\bm,j); i} -t'_j
        \rhoal_{(\bm,j);ij}\Cal_{(\bm,j);j}\\ 
      &= \lp 1+t'_j\rhoal_{(\bm,j);jj} \rp \lp \lp 1+t_j\rhoal_{\bm;jj}
        \rp \Cal_{\bm;i}-t_j \rhoal_{\bm;ij}\Cal_{\bm;i} \rp \\ 
      &\quad -t'_j \left(\rhoal_{\bm; ij}-\frac{t_j \rhoal_{\bm; jj}\rhoal_{\bm; ij}}{1+t_j\rhoal_{\bm; jj}}\right) \Cal_{\bm;j}\\
      &= \lp 1+(t_j+t_j')\rhoal_{\bm;jj} \rp
        \Cal_{\bm;i}-(t_j+t_j')\rhoal_{\bm;ij} \Cal_{\bm;j}\\ 
      &=\Cal_{(\bm,j);i}(z;(\bt_\bm, t_j+t_j')).
  \end{align*}
\end{proof}

\section{Examples}
\label{sec:examples}

We conclude by showing some explicit examples of exceptional Gegenbauer
polynomials of the second kind, together with their properties. It can be readily checked that these families are an isospectral deformation of the classical Gegenbauer polynomials.

\subsection{1-parameter exceptional Gegenbauer polynomials of the second kind} \hfill

The 1-parameter exceptional Gegenbauer polynomials arise after a
single CDT on the classical operator. For
$m \in \Nz$,  we follow definitions
\eqref{eq:taum}--\eqref{eq:pimi} to write
\begin{align*}
    &\taual_{m} (z, t_{m}) = 1 + t_{m} \rhoal_{m m}(z) , \\
    &\Cal_{m;i} (z, t_{m}) = \lp 1 + t_{m} \rhoal_{m m}(z) \rp \Cal_{i}(z) - t_{m} \rhoal_{i m}(z) \Cal_{m}(z) ,
\end{align*}
with $\rhoal_{ij}$ as per \eqref{eq:Rijdef} and $\Cal_i$ the classical
Gegenbauer polynomials. Notice that by construction, we have
$\Cal_{m;m} = \Cal_{m}$.  The $\{ \Cal_{m;i} \}_{i \in \Nz}$ is a
family of exceptional Gegenbauer polynomials with weight
\[ \Wal_m(z) = \frac{(1-z^2)^{\al-\frac12}}{ \big(\taual_{m}(z)\big)^{2}},\]
as long as $t_{m}$ satisfies the inequality
\begin{gather*}
  t_{m} > - \lp\nual_{m}\rp^{-1} = - \frac{2^{2\al-1} m! (m+\al)
    \Gamma(\al)^2 }{\pi \Gamma(m+2\al)}.
\end{gather*}
The orthogonality relations are
\begin{align*}
  &\int_{-1}^{1} \Cal_{m;i} (z,t_{m}) \Cal_{m;j} (z,t_{m})
    \Wal_m(z) dz  =\delta_{ij} \nual_i =  \delta_{ij} \frac{\pi 2^{1-2\alpha}
    \Gamma(i+2\alpha)}{i! (i+\alpha) \Gamma(\alpha)^2} ,\quad i, j \in
    \Nz \backslash \{m\} \\ 
  &\int_{-1}^{1} \Cal_{m;m} (z,t_{m})^2 \Wal_m(z,t_m)dz =
    \frac{\nual_m}{1+ t\nual_m} = \frac{\pi
    2^{1-2\alpha} \Gamma(m+2\alpha)}{m! (m+\alpha) \Gamma(\alpha)^2 +
    t_{m} \pi 2^{1-2\alpha} \Gamma(m+2\alpha)} . 
\end{align*}

The function $\taual_{m}$ is shown below in Figure ~\ref{fig:1}, for a particular choice of the parameters $\alpha, m, t_{m}$.
\begin{figure}[ht]
    \centering
    \includegraphics[width=0.5\textwidth]{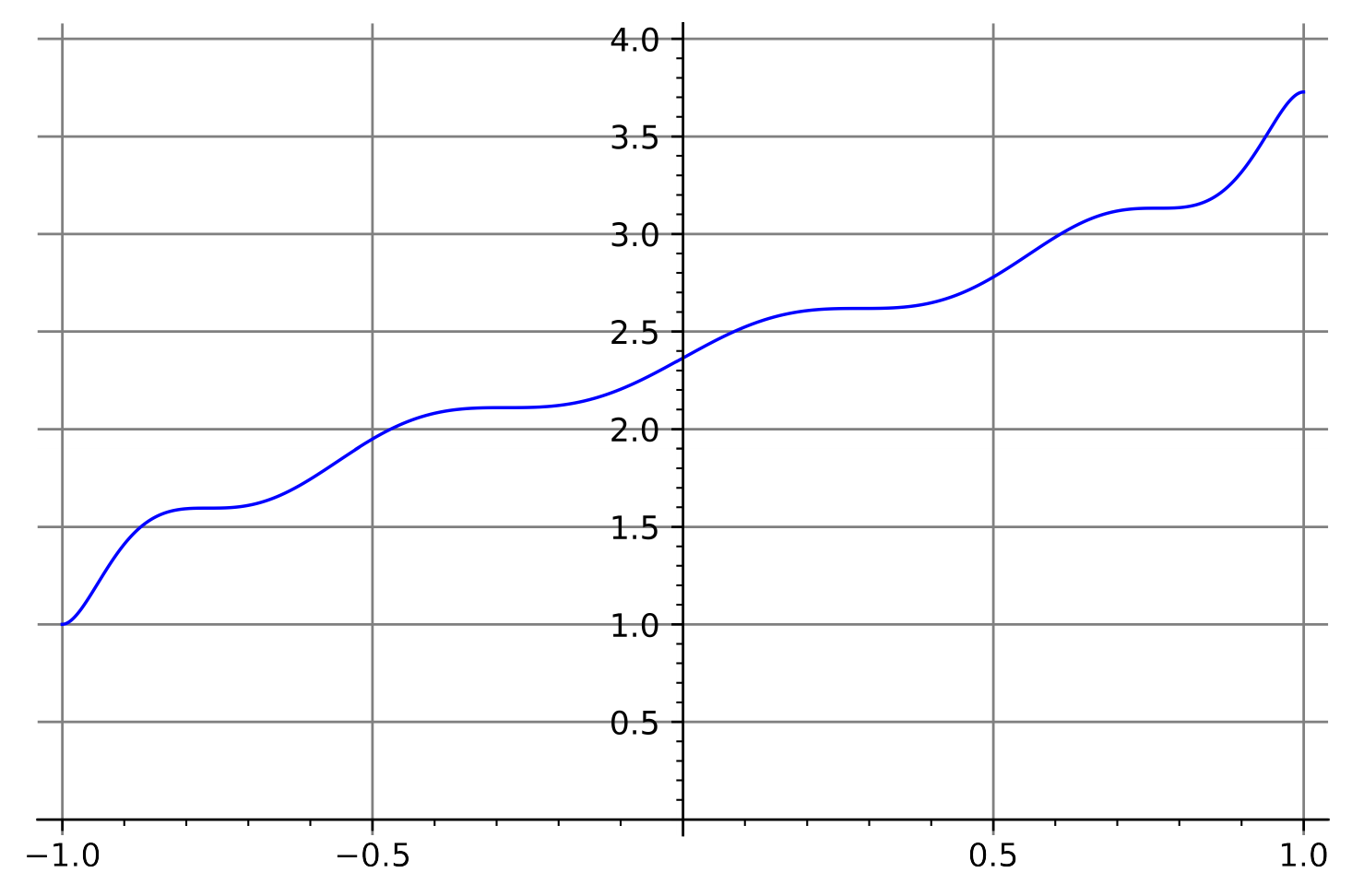}
    \caption{The function $\taual_{m} (z, t_{m})$ for $m=4$ and
      $t_{m}=0.5$.}
    \label{fig:1}
\end{figure}
Clearly, there are no zeroes. Figure ~\ref{fig:2} shows the polynomial families for $\al=3/2$, $m=2$, and different values of $t_{m}$.

\begin{figure}[h]
    \centering
    \begin{tabular}{cc}
    \includegraphics[width=0.45\textwidth]{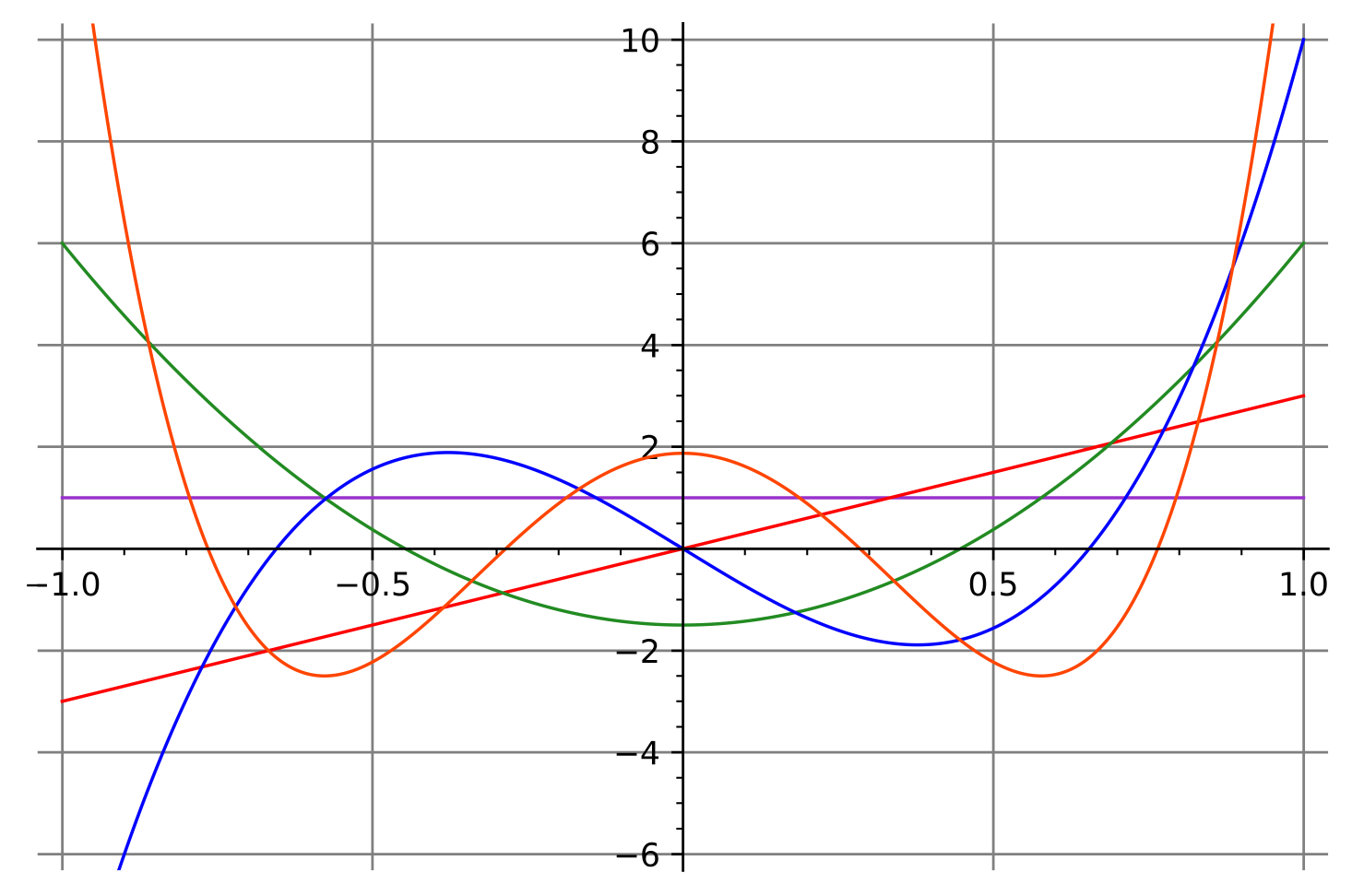} & \includegraphics[width=0.45\textwidth]{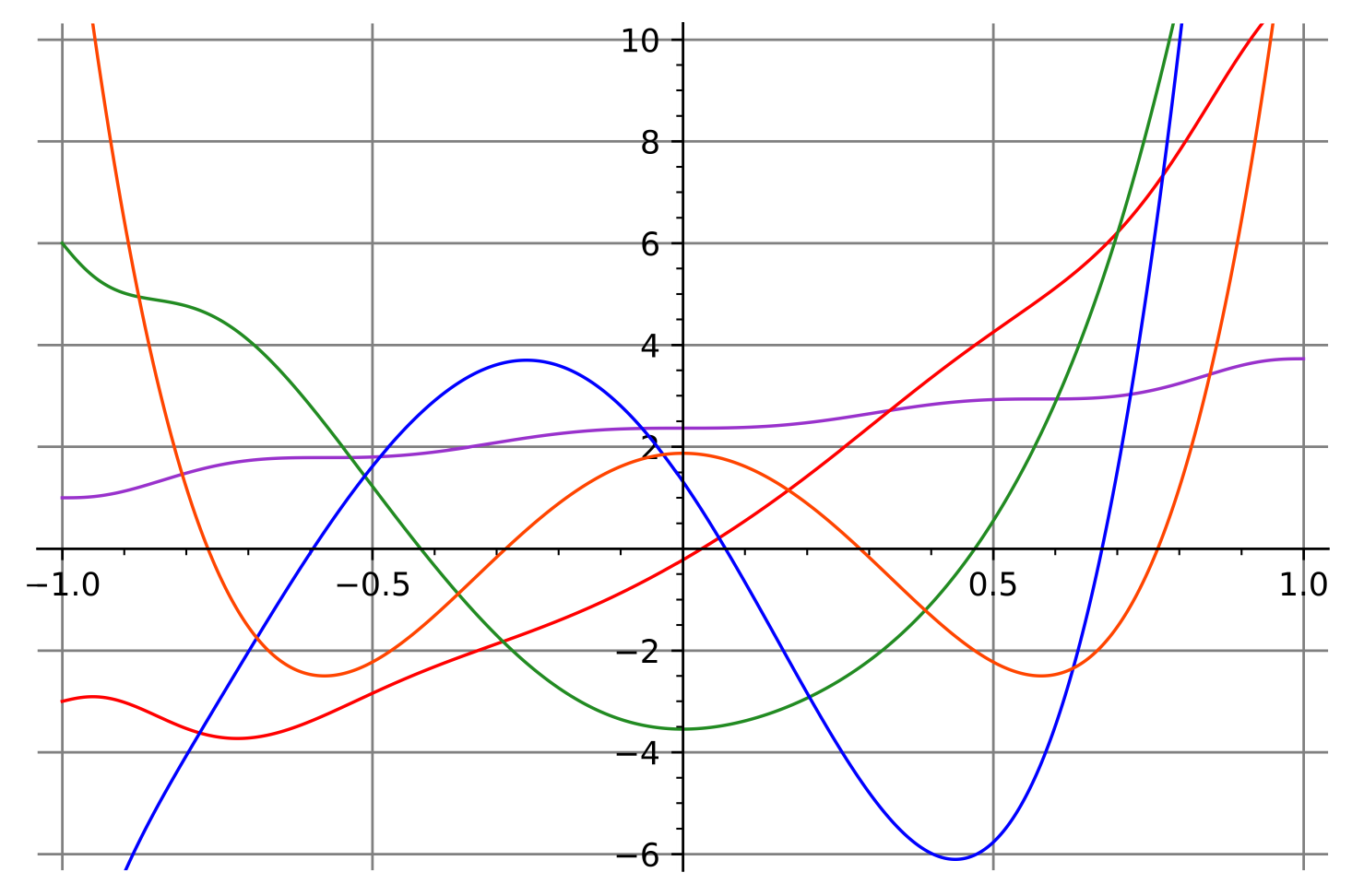}
    \end{tabular}
    \caption{First few exceptional Gegenbauer polynomials $\Cal_{m;i}(z;t_{m})$ 
    for $m=4$, with $t_{m}=0$ (left) and $t_{m}=0.5$ (right).}
    \label{fig:2}
\end{figure}

The first few polynomials for $m=4$ and $\al=3/2$ are explicitly given by
\begin{align*}
    &C^{(3/2)}_{4;0} (z, t_4) = C^{(3/2)}_{0}(z) + \frac{15}{176} t_4 \lp 945 z^{11} - 3080 z^9 + 3630 z^7 - 1848 z^5 + 385 z^3 + 32 \rp , \\
    &C^{(3/2)}_{4;1} (z, t_4) = C^{(3/2)}_{1}(z) + \frac{45}{5632} t_4 \big( 19845 z^{12} - 59290 z^{10} + 59455 z^8 \\
    &\qquad\qquad\qquad\qquad\qquad\qquad\qquad\qquad\qquad - 20636 z^6 + 275 z^4 + 1430 z^2 + 1024 z - 55 \big) , \\
    &C^{(3/2)}_{4;2} (z, t_4) = C^{(3/2)}_{2}(z) + \frac{45}{704} t_4 \big( 3675 z^{13} - 11515 z^{11} + 13310 z^9 - 7590 z^7 \\
    &\qquad\qquad\qquad\qquad\qquad\qquad\qquad\qquad\qquad + 2871 z^5 - 495 z^3 + 320 z^2 - 64 \big) , \\
    &C^{(3/2)}_{4;3} (z, t_4) = C^{(3/2)}_{3}(z) + \frac{75}{2816} t_4 \big( 9261 z^{14} - 30919 z^{12} + 39501 z^{10} - 24783 z^8 \\
    &\qquad\qquad\qquad\qquad\qquad\qquad\qquad\qquad\qquad + 7007 z^6 + 2739 z^4 + 1792 z^3 - 1881 z^2 - 768 z + 99 \big) , \\
    &C^{(3/2)}_{4;4} (z, t_4) = C^{(3/2)}_{4}(z) .
\end{align*}

\noindent
Notice that, as illustrated by Figure ~\ref{fig:2}, the exceptional polynomials are continuous deformations of the corresponding classical polynomials, aside from $C^{(3/2)}_{4;4}$.

\subsection{2-parameter exceptional Gegenbauer polynomials of the second kind} \hfill

We construct the 2-parameter exceptional polynomials by applying the
recursive construction to the 1-parameter formulas. We start with
$\bm = (m_1, m_2) \in \Nz^2$ and the associated tuple
$\bt_{\bm} =(t_{m_1}, t_{m_2})$. Following equations
\eqref{eq:rhoij}--\eqref{eq:pimi}, and using \eqref{eq:rho2step}, we
find that
\begin{align*}
    &\taual_{\bm} (z; \bt_{\bm}) = \taual_{m_1} (z, t_{m_1}) \taual_{m_2} (z, t_{m_2}) - t_{m_1} t_{m_2} \rhoal_{m_1 m_2}(z)^2 , \\
    &\Cal_{\bm;i} (z; \bt_{\bm}) = \Cal_{i}(z) \taual_{\bm} (z; \bt_{\bm}) - t_{m_1} \Cal_{m_1}(z) \taual_{m_2}(z; t_{m_2}) \rhoal_{m_2; m_1 i} (z, t_{m_2}) \\
    &\qquad\qquad\qquad - t_{m_2} \Cal_{m_2}(z) \taual_{m_1}(z; t_{m_2}) \rhoal_{m_1; m_2 i} (z, t_{m_1}) ,
\end{align*}
where
\begin{align*}
  \rhoal_{m; i j}(z, t_{m})
  &= \int_{-1}^z \Cal_{m; i} (u, t_{m})    \Cal_{m; j} (u, t_{m}) \Wal_m(u) du\\   
  &= \rhoal_{ij}(z) - \frac{t_{m} \rhoal_{m i}(z) \rhoal_{m j}(z)}{1 +
    t_{m} \rhoal_{m m}(z)}, \qquad m,i, j \in \Nz. 
\end{align*}
Once again, the polynomials form a complete orthogonal basis relative
to the weight
\[ \Wal_\bm(z) = (1-z^2)^{\al-\frac12} \taual_{m_1 m_2}(z;t_{m_1},
  t_{m_2})^{-2},\] provided
\[
  t_{m_i} > - \frac{2^{2\al-1} m_i! (m_i+\al) \Gamma(\al)^2 }{\pi
      \Gamma(m_i+2\al)},\quad i=1,2.
\]
The orthogonality relations are
\begin{align*}
  &\int_{-1}^{1} \Cal_{\bm;i} (z, \bt_{\bm}) \Cal_{\bm;j} (z,
    \bt_{\bm}) \Wal_\bm(z,\bt_\bm) dz = \delta_{ij} \frac{\pi
    2^{1-2\alpha} \Gamma(i+2\alpha)}{i! (i+\alpha) \Gamma(\alpha)^2}
    ,\quad i, j \in \Nz \backslash \{m_1,m_2\}, \\ 
  &\int_{-1}^{1} \Cal_{\bm;m_i} (z, \bt_{\bm})^2
    \Wal_\bm(z,\bt_\bm) dz =    \frac{\pi 2^{1-2\alpha}
    \Gamma(m_i+2\alpha)}{m_i! (m_i+\alpha) 
    \Gamma(\alpha)^2 + t_{m_i} \pi 2^{1-2\alpha} \Gamma(m_i+2\alpha)}
    , \quad i=1,2.
\end{align*}

\bigskip
\subsection*{Acknowledgements}
Mar\'ia­ \'Angeles Garc\'ia-Ferrero was supported  by the Spanish MINECO through Juan de la Cierva fellowship FJC2019-039681-I, by the Spanish State Research Agency through BCAM Severo Ochoa excellence accreditation SEV-2017-0718 and by the Basque Government through the  BERC Programme 2018-2021. The research of David G\'omez-Ullate has been financed by projects PGC2018-096504-B-C33 and RTI2018-100754-B-I00 from  FEDER/Ministerio de Ciencia e Innovación– Agencia Estatal de Investigación,  the European Union under the 2014-2020 ERDF Operational Programme and the Department of Economy, Knowledge, Business and University of the Regional Government of Andalusia (project FEDER-UCA18-108393).


\end{document}